\documentclass[12pt]{amsart}

\usepackage{amssymb}

\usepackage{enumitem}

\usepackage{graphicx}

\makeatletter
\@namedef{subjclassname@2020}{%
  \textup{2020} Mathematics Subject Classification}
\makeatother

\usepackage[T1]{fontenc}

\newtheorem{theorem}{Theorem}[section]
\newtheorem{Corollary}[theorem]{Corollary}
\newtheorem{lemma}[theorem]{Lemma}
\newtheorem{proposition}[theorem]{Proposition}
\newtheorem{problem}[theorem]{Problem}
\newtheorem{Fact}[theorem]{Fact}
\newtheorem{Construction}[theorem]{Construction}
\newtheorem{Conclusion}[theorem]{Conclusion}

\newtheorem{mainthm}[theorem]{Main Theorem}

\theoremstyle{definition}
\newtheorem{definition}[theorem]{Definition}
\newtheorem{remark}[theorem]{Remark}

\newtheorem*{xrem}{Remark}

\numberwithin{equation}{section}

\def\includeE#1{{\lhook\kern-3.5pt\joinrel\smash{
    \mathop{\longrightarrow}\limits^{#1}}}}

\def\efor/{Example~\ref{E4}}

\def\BL/{Baldwin--Lachlan}
\def\Bu/{Buechler}
\def\Hr/{Hrushovski}
\def\lm/{locally modular}
\def\wm/{weakly minimal}
\def\nm/{non--modular}
\def\ss/{superstable}
\def\ud/{unidimensional}
\def\sm/{strongly minimal}

\def\abar{\overline{a}}

\def\bbar{\overline{b}}
\def\cbar{\overline{c}}

\def\hbar{\overline{h}}

\def\xbar{\overline{x}}
\def\ybar{\overline{y}}

\def\dom{{\rm dom}}

\def\tp{{\rm tp}}

\def\tr/{trivial}
\def\nt/{non--trivial}
\def\st/{strong type}

\def\abar{\bar{a}}
\def\bbar{\bar{b}}
\def\cbar{\bar{c}}

\def\ybar{\bar{y}}
\def\phi{\varphi}

\def\B{{\mathcal B}}
\def\C{{\mathcal  C}}

\def\F{{\mathcal F}}
\def\FF{{\bf F}}
\def\M{{\mathcal M}}

\def\P{{\mathcal P}}
\def\Q{{\mathbb Q}}

\def\tp{{\rm tp}}

\def\dom{{\rm dom}}

\def\dcl{{\rm dcl}}

\def\Fa0{{\FF^a_{\aleph_0}}}

\def\<{\langle}
\def\>{\rangle}

\def\V{{\mathbb V}}

\def\Lsharp{{L^\sharp}}
\def\Msharp{{M^\sharp}}
\def\Lflat{{L^{\flat}}}
\def\Mflat{{M^{\flat}}}
\def\dc{{\rm dc}}
\def\B{{\mathbb B}}
\def\aa{{\bf a}}
\def\bb{{\bf b}}
\def\cc{{\bf c}}
\def\dd{{\bf d}}
\def\ee{{\bf e}}

\def\ss{{\bf s}}

\def\FM{{\rm Fm}}
\def\Fm{{\rm Fm}}
\def\qftp{{\rm qftp}}
\def\ptl{{\rm ptl}}
\def\flattp{{\flat\hbox{-}\tp}}

\def\E{{\mathcal E}}
\def\conj{{\rm conj}}
\def\cmap{{\rm cmap}}
\DeclareMathOperator{\Mod}{Mod}
\DeclareMathOperator{\HC}{HC}
\DeclareMathOperator{\Hdf}{Hdf}

\DeclareMathOperator{\Flat}{Flat}
\DeclareMathOperator{\Aut}{Aut}
\DeclareMathOperator{\Fix}{Fix}
\DeclareMathOperator{\Sym}{Sym}
\DeclareMathOperator{\CSS}{CSS}
\DeclareMathOperator{\css}{css}

\DeclareMathOperator{\can}{can}
\DeclareMathOperator{\bcan}{\can^\flat}
\DeclareMathOperator{\Kn}{Kn}
\DeclareMathOperator{\Succ}{Succ}
\def\Ax{{\rm Ax^\flat}(L)}
\def\Axprime{{\rm Ax^{\flat}}(L')}
\def\Axstar{{\rm Ax^{\flat}}(L^*)}
\def\Axsharp{{\rm Ax^{\sharp}}(L)}
\def\QF{{\rm QF}}
\def\emptyseq{\<\>}
\def\const{{\bf c_{\emptyseq}}}

\newcommand\myrestriction{\mathord\restriction}
\def\mr#1{\myrestriction_{#1}}

\DeclareMathOperator{\len}{length}

\newcommand{\obj}[3]{\mathcal{F}^{#1}\mathbb{S}^{#2}\mathbf{G}_{#3}}

\begin{document}


\baselineskip=17pt


\title[The existence of Borel complete expansions]{Characterizing the existence of a Borel complete expansion}

\author[M. C. Laskowski]{Michael C. Laskowski}
\address{Department  of Mathematics\\ University of Maryland\\
College Park, MD 20742  USA}
\email{laskow@umd.edu}

\author[D. S. Ulrich]{Douglas S. Ulrich}
\address{Department  of Mathematics\\ University of Maryland\\
College Park, MD 20742 USA}
\email{ds\_ulrich@hotmail.com}

\date{}


\subjclass[2020]{Primary 03C50; Secondary 03E15}

\keywords{Borel complexity, flat structures, potential cardinality}

\begin{abstract}  
	We develop general machinery to cast the class of potential canonical Scott sentences of an infinitary sentence $\phi$ as a class of structures in a related language.
	We  show that 
	$\phi$ has a Borel complete expansion if and only if $S_\infty$ divides $\Aut(M)$ for some countable model $M\models \phi$.
	From this, we prove that  for theories $T_h$  asserting that $\{E_n\}$ is a countable family of cross cutting equivalence relations with $h(n)$ classes, if
	$h(n)$ is uniformly bounded, then $T_h$ is not Borel complete, providing a converse to Theorem~2.1 of \cite{LU}.
	\end{abstract}

\maketitle


\section{Introduction}  
Back and forth systems are an invaluable tool in understanding the descriptive set theoretic complexity of a class of models.  Indeed, within the realm of countable structures,
back and forth equivalence is the same as isomorphism.  As every structure can be made countable in a forcing extension $\V[G]$ of the the set-theoretic 
universe $\V$, two arbitrary structures in the same vocabulary are back and forth equivalent if and only if they are potentially isomorphic, i.e., they can become isomorphic in some forcing extension.

Classically, given a countable structure $M$, there is a `preferred back and forth system' $\F_\infty:=\{(\abar,\bbar)\in M^{2n}: (M,\abar)\equiv_{\infty,\omega} (M,\bbar)\}$ and the data of
$(M,\F_\infty)$ is coded into a {\em Scott sentence} $\phi_M\in L_{\omega_1,\omega}$, whose countable models $N$ are precisely those isomorphic to $M$.  This assignment can be done canonically, giving a canonical object $\css(M)$.  Similarly, a structure $M$ of any size has an associated canonical Scott sentence $\css(M)$, which is a sentence of $L_{\infty,\omega}$ 
that describes the back-and-forth class of $M$.  Extending on this correspondence, given a sentence $\Phi\in L_{\omega_1,\omega}$, we can look at 
$\CSS(\Phi)_{{\rm sat}}:=\{\css(M):M\models \Phi\}$, which may be a proper class, but is bijective with the class $\Mod(\Phi)/\equiv_{\infty,\omega}$.  However, for applications, a possibly larger class $\CSS(\Phi)_\ptl$ is useful.   This class consists of sentences $\psi\in\V$, $\psi\in L_{\infty,\omega}$, such that in every forcing extension $\V[G]\supseteq\V$, if $(\psi\in L_{\omega_1,\omega})^{\V[G]}$, then $\psi=\css(N)$ for some $N\models \Phi$ in $\V[G]$.  In \cite{URL}  sentences in $\CSS(\Phi)_\ptl$  are called the {\em potential canonical Scott sentences} of models of $\Phi$.  

The number of potential canonical Scott sentences is relevant to the complexity of  $(\Mod_\omega(\Phi),\cong)$, the set of models of $\Phi$ with universe $\omega$, modulo isomorphism.
In \cite{URL} it is proved that if there is a Borel reduction $(\Mod_\omega(\Phi),\cong)\le_B (\Mod_\omega(\Psi),\cong)$, then $|\CSS(\Phi)_\ptl|\le |\CSS(\Psi)_\ptl|$ (where we interpret these `cardinalities' to be $\infty$ if they are proper classes).  
This is quite useful, although in the above formulation
 the elements of $\CSS(\Phi)_\ptl$ are simply sentences of $L_{\infty,\omega}$ and are hard to manipulate.

There are two parts to this paper.  In the first, we obtain a natural bijection between $\CSS(\Phi)_\ptl$ and a class of structures in an associated language $\Lflat$.
 In Section~2 we begin by introducing ``sharp" back and forth systems and structures
and  describe a canonical {\em flattening} operation that associates an $\Lflat$-structure 
$\B=(M,\F)^\flat$ to each sharp structure $(M,\F)$.   We show that the flattening operation transforms (infinitary) $L$-formulas $\phi$ into (infinitary) $\Lflat$-formulas $\phi^\flat$
in a natural way.  Using this, we distinguish the class of {\em Hausdorff} flat structures that correspond to flattenings of the `preferred'  sharp structures $(M,\F_{\infty})$.
In Section~\ref{Hdfstuff} we give a list $\Ax$ of infinitary $\Lflat$-sentences that
axiomatize these flattenings.  We show that for a countable flat structure $\B$, we can recover a sharp structure $(M,\F)$ whose flattening is isomorphic to $\B$.
We show that   the class of Hausdorff flat structures is absolute between forcing extensions 
and we define a canonical class of representatives $\Hdf^*$ of the $\Lflat$-isomorphism types of Hausdorff flat structures.  We  see that this class $\Hdf^*$ is naturally bijective with
$\CSS(L)_\ptl$, so we propose to take $\Hdf^*$ as an improved representation of potential canonical Scott sentences.  All of this relativizes to models of a given $\Phi\in L_{\omega_1,\omega}$, yielding
a class
$\Hdf^*(\Phi)$ that is naturally bijective with $\CSS(\Phi)_\ptl$.

The second part of the paper exploits the fact that the class $\Flat(\phi)$ of flat models of $\phi^\flat$ is axiomatized by the sentence $\phi^\flat\wedge\Ax\in (L^\flat)_{\omega_1,\omega}$.  Non-Hausdorff flat models correspond to flattenings of sharp structures $(M,\F)$ where $\F$ is something other than $\F_\infty$, which in turn codes the structure of an expansion $M^*$ of $M$.  In Section~\ref{precise} we make this correspondence precise and then in Section~\ref{BCexpansions}, by 
applying Morley's theorem to the $(L^\flat)_{\omega_1,\omega}$-sentence
$\phi^\flat\wedge\Ax$, we see that the existence of a Borel complete expansion of a model of $\phi$ is equivalent to Hjorth's notion of `$S_\infty$ divides $\Aut(M)$' for some countable
$M\models\phi$.\footnote{In \cite{Hjorth}, Hjorth proved that $\Aut(M)$ fails the Topological Vaught Conjecture on analytic sets if and only if $S_\infty$ divides $\Aut(M)$.}
This connection is indirect, in that we prove that each of these properties is equivalent to $\phi^\flat\wedge\Ax$ having arbitrarily large models, which in turn
is equivalent to this infinitary theory having Ehrenfeucht-Mostowski models.   The following result is proved as Theorem~\ref{HjorthThm} in Section 5.

\begin{theorem}  \label{HjorthThm1.1}  The following are equivalent for a sentence $\phi\in L_{\omega_1,\omega}$.
\begin{enumerate}
\item  $\phi$ has a Borel complete expansion (i.e., there is some countable $L'\supseteq L$ and $\psi\in (L')_{\omega_1,\omega}$ such that
$\psi\vdash \phi$ and $\psi$ is Borel complete);
\item  $\phi^\flat\wedge\Ax$ has arbitrarily large models;
\item  $\phi^\flat\wedge\Ax$ admits Ehrenfeucht-Mostowski models;
\item  $S_\infty $ divides $\Aut(M)$ for some countable $M\models \phi$.
\end{enumerate}
\end{theorem}

In Section 6 we discuss consequences of 
Theorem~\ref{HjorthThm1.1}.  We get a distinction between the Borel complexity of first order theories and sentences of $L_{\omega_1,\omega}$.
In particular, with Corollary~\ref{example}
we see that every first order theory with an infinite model has a Borel complete expansion, whereas there are are sentences of $L_{\omega_1,\omega}$ (even complete ones) that do not.  
One example of an infinitary sentence without a Borel complete expansion is the sentence $\phi_h$ that is used in the proof of Theorem~\ref{bounded}.   There it is proved that 
 the theory of cross-cutting equivalence relations $T_h$ in the language $\{E_n:n\in\omega\}$
with a uniform, finite bound on the number of $E_n$-classes is not Borel complete.  This is in contrast to Theorem~2.1 of \cite{LU}, where it is proved that the analogous theory
with unboundedly many classes is Borel complete.


\subsection{Preliminaries}

Following \cite{URL}, let $HC$ denote the set of hereditarily countable sets.  We begin by defining a class of formulas whose restrictions to HC are well behaved.

\begin{definition}  \label{HCdef}
Suppose $\phi(x)$ is any first order formula of set theory, possibly with a hidden parameter from $\HC$.  
\begin{itemize}
\item  $\phi(\HC)=\{a\in \HC:(\HC,\in)\models \phi(a)\}$ and if
 $\V[G]$ is a forcing extension of $\V$, then $\phi(\HC)^{\V[G]}=\{a\in \HC^{\V[G]}:\V[G]\models 
\hbox{`$a\in \phi(\HC)$'}\}$.
\item  $\phi(x)$ is
{\em $\HC$-forcing invariant}
if, for every twice-iterated forcing extension $\V[G][G']$,
\[\phi(\HC)^{\V[G][G']}\quad \cap\quad  \HC^{\V[G]}\quad = \quad \phi(\HC)^{\V[G]}\]
\end{itemize}

\end{definition}

The reader is cautioned that when computing $\phi(\HC)^{\V[G]}$, the quantifiers of $\phi$ range over $\HC^{\V[G]}$ as opposed to the whole of $\V[G]$.
The reason for the double iteration of forcing is to make the notion of a formula $\phi(x)$ being HC-invariant absolute between forcing extensions.
Visibly, the class of $\HC$-forcing invariant formulas is closed under boolean combinations, and it follows from L\'evy Absoluteness that every $\Sigma_1$ formula of set theory is $HC$-invariant,
see e.g., Lemma~2.2 of \cite{URL}.

The principal idea we wish to exploit is that every set is potentially in HC.  That is, for any set $A\in\V$, there is a forcing extension $\V[G]\supseteq\V$ such that $A\in HC^{\V[G]}$.
With this in mind, for any HC-invariant formula $\phi(x)$, let
$$\phi_\ptl=\{A\in\V:\V[G]\models \phi(A)\ \hbox{for some/every $\V[G]\supseteq \V$ with $A\in HC^{\V[G]}\}$}$$
Associated to every HC-invariant $\phi(x)$ is a  {\em strongly definable} family $X=(X^{\V[G]})$, indexed by collection of forcing extensions $\V[G]\supseteq\V$, where each $X^{\V[G]}=\phi(HC)^{\V[G]}$.
We say $X$ is {\em strongly definable} if it is strongly definable via some HC-invariant formula.
For $X$ strongly definable, we write $X_{\ptl}$ for $\phi_\ptl$, where $\phi(x)$ is any HC-invariant formula defining $X$.
We call $X$ {\em short} if $X_\ptl$ is a set, as opposed to a proper class.
For a short $X$, we let $||X||:=|X_\ptl|$,   and we write $||X||=\infty$ when $X$ is not short.

If a strongly definable $X$ is defined by a formula $\phi(x)$ that is absolute between any forcing extensions, then, using the fact that
every set $A\in\V$ is in $HC^{\V[G]}$ for some forcing extension $\V[G]\supseteq\V$,  $X_\ptl$ is all sets $A\in\V$ such that $\phi(A)$ holds, i.e.,
$A\in X^{\V[G]}$ for some forcing extension $\V[G]\supseteq\V$.
We abuse notation slightly and refer to this class as $X$. 

As examples, for any countable language $L$ and any $\Phi\in L_{\omega_1,\omega}$,  both $\Mod_{HC}(L)=\{M\in HC:  M$ is an $L$-structure$\}$ 
and $\Mod_{HC}(\Phi)=\{M\in\Mod_{HC}(L):M\models\Phi\}$ are 
strongly definable.  
Here, as being an $L$-structure and satisfaction are absolute,
we write $\Mod(L)$ and $\Mod(\Phi)$ in place of the more cumbersome $(\Mod_{HC}(L))_\ptl$ or $(\Mod_{HC}(\Phi))_\ptl$.


Working with strongly definable sets, we say a notion holds {\em persistently} if it holds in every forcing extension $\V[G]\supseteq\V$.  As examples,
we say $f:X\rightarrow Y$ persistently if, for every $\V[G]\supseteq\V$, $f^{\V[G]}$ is a function with domain $X^{\V[G]}$, taking elements in $Y^{\V[G]}$.
{\em Persistently injective} and {\em  persistently bijective} functions are defined similarly.  
We record the following fact, which appears as Lemma~2.14 of \cite{URL}.

\begin{lemma}  \label{ptl}  Suppose $f,A,B$ are strongly definable with $A,B\subseteq HC$
and $f:A\rightarrow B$ is persistently a function.  Then $f_{\ptl}$ is a class functional from $A_{\ptl}$ to $B_{\ptl}$.

If, in addition, $f:A\rightarrow B$ is persistently injective (resp.\ bijective) then $f_\ptl$ is injective (resp.\ bijective).
%
\end{lemma}

We also include one new notion that was not discussed in \cite{URL}.  For a persistent function $f:A\rightarrow B$ as in Lemma~\ref{ptl}, we say {\em $f$ has a persistent cross section} if there is a strongly definable $g:B\rightarrow A$ such that, persistently, $(\forall b\in B) f\circ g(b)=b$. 

\begin{lemma}  \label{cross}  Suppose $f,A,B$ are strongly definable with $A,B\subseteq HC$
and $f:A\rightarrow B$ is persistently a function.  If $g:B\rightarrow A$ is a persistent cross section of $f$, then $f_\ptl:A_\ptl\rightarrow B_\ptl$ is surjective.
\end{lemma}

\begin{proof}  Given $b\in B_\ptl$, put $a:=g_\ptl(b)$.  Then $f_\ptl(a)=b$, so $f_\ptl$ is surjective.
\hfill\qedsymbol\end{proof}


Much of Descriptive Set Theory, or at least the study of Borel reducibility, revolves around analyzing the complexity of equivalence relations, and hence quotients.
Call a pair $(X,E)$ a {\em strongly definable quotient} if both $X$ and $E$ are strongly definable and $E$ is persistently an equivalence relation on $X$.
An {\em HC-reduction} of $(X,E)$ into another strongly definable quotient $(Y,F)$ is a strongly definable 
$f\subseteq X\times Y$ such that, persistently,
\begin{itemize}
		\item The $E$-saturation of $\dom(f)$ is $X$.  That is, for every $x\in X$, there is an $x'\in X$ and $y'\in Y$ where $xEx'$ holds and $(x',y')\in f$.
		\item $f$ induces a well-defined injection on equivalence classes.  That is, if $(x,y)$ and $(x',y')$ are in $f$, then $xEx'$ holds if and only if $yFy'$ does.
	\end{itemize}

We say  {\em $(X,E)$ is HC-reducible to $(Y,F)$,} written
$$(X,E)\le_{HC} (Y,F)$$
if an HC-reduction of $(X,E)$ into $(Y,F)$ exists, and we say that $(X,E)$ and $(Y,F)$ are {\em HC-bireducible} if
if $(X,E)$ is HC-reducible to $(Y,F)$ and $(Y,F)$ is HC-reducible to $(X,E)$.
As examples, if $X,E,Y,F$ are  Borel subsets of Polish spaces then any Borel reduction is an $HC$-reduction, as is any absolutely $\Delta^1_2$-reduction, see e.g., \cite{URL}. 

 A {\em representation} of a strongly definable quotient  $(X,E)$ is
a pair $f,Z$ of strongly definable families such that  $f:X\rightarrow Z$ is persistently
surjective and persistently,
$$(\forall a,b\in X) [E(a,b)\leftrightarrow f(a) = f(b)]$$
It is easily checked that if $f,Z$ and $f',Z'$ are two representations of the same $(X,E)$, then $||Z||=||Z'||$, so provided a representation exists, 
we define $||(X,E)||$ to be this common value, which may be $\infty$.  
We call $(X,E)$ {\em short} if $||(X,E)||<\infty$. 
The following fact is essentially Proposition~2.18 of \cite{URL}.

\begin{Fact}  \label{bound}  Suppose $(X,E)$ and $(Y,F)$ are both represented, strongly definable quotients.
If  $(X,E)\le_{HC}(Y,F)$, then $||(X,E)||\le ||(Y,F)||$.
\end{Fact}

Continuing with our example above, for a countable language $L$, let $\Mod_\omega(L)$ be the Polish space of $L$-structures with universe $\omega$.
By a theorem of Lopez-Escobar, \cite{LE}, the Borel subsets of $\Mod_\omega(L)$, invariant under the action of $S_\infty$, are of the form $\Mod_\omega(\Phi)$
for some $\Phi\in L_{\omega_1,\omega}$.
In \cite{FS}, Friedman and Stanley introduced the concept of Borel reducibility between isomorphism classes of this form.  
For infinitary $\Phi,\Psi$, possibly in different languages, $(\Mod_\omega(\Phi),\cong)$ is {\em Borel reducible} to $(\Mod_\omega(\Psi),\cong)$ if there is a Borel
$f:\Mod_\omega(\Phi)\rightarrow \Mod_\omega(\Psi)$ satisfying $M\cong N$ iff $f(M)\cong f(N)$ for all $M,N\in\Mod_\omega(\Phi)$.
They showed that among invariant Borel subsets, there is a maximal class, dubbed  {\em Borel complete}, with respect to Borel reducibility.  They showed that  e.g., the classes
of graphs, linear orders, and trees are Borel complete.  
 
Placing this into our context,
 there is a representation
$$\css:(\Mod_{HC}(L),\cong)\rightarrow \CSS(L)$$
where, for each $L$-structure $M\in HC$, $\css(M)$ is the {\em canonical Scott sentence of $M$}, which is a sentence of $L_{\omega_1,\omega}$ (see Definition~3.1 of \cite{URL} for a formal
definition of this map).  The salient property of $\css(M)$ is described by

\begin{Fact}  \label{later}  For any countable language $L$ and any $M,N\in\Mod_{HC}(L)$, $N\models \css(M)$ if and only if $N\cong M$.
\end{Fact}

That the set $\CSS(L)$ is strongly definable and that $\css$ is a persistent HC-invariant surjection  follows from Karp's Completeness Theorem for sentences of
$L_{\omega_1,\omega}$ (see e.g., Theorem~3 of \cite{Keisler} and Lemma~3.3 of \cite{URL}).  

%

For any $\Phi\in L_{\omega_1,\omega}$, the restriction of the $\css$ representation gives
a representation $\CSS(\Phi)$ of $(\Mod_{HC}(\Phi),\cong)$.  Thus, 
$\CSS(\Phi)_\ptl$, the class of {\em potential canonical Scott sentences} is the class of all $\psi\in\V$, where $\psi\in L_{\infty,\omega}$, but for some/every 
 $\V[G]\supseteq \V$ where $(\psi\in L_{\omega_1,\omega})^{\V[G]}$, we have $(\psi\in \CSS(\Phi))^{\V[G]}$.
 
 Note that $\CSS(\Phi)$ is also a representation of $(\Mod_\omega(\Phi),\cong)$, the isomorphism classes of models of $\Phi$ with universe $\omega$.
 We write $||\Phi||$ in place of $||(\Mod_{HC}(\Phi),\cong)||=|\CSS(\Phi)_\ptl|$ (which may be $\infty$).   If $(\Mod_\omega(\Phi),\cong)$ is Borel complete,
 then $\Phi$ is not short, i.e., $||\Phi||=\infty$.  This follows from the fact that there is a proper class of back-and-forth inequivalent graphs.

Note, however, that for some $\Phi$, the class $\CSS(\Phi)_\ptl$  can be strictly larger than $\CSS(\Phi)_{{\rm sat}}$, the latter being the class of all $\css(M)$ for $M\in \Mod(\Phi)$.
When these two classes are not equal, we say that {\em $\Phi$ is not grounded}.  
One of the central goals of Section~\ref{Hdfstuff} is to get better control of the class of potential canonical Scott sentences of $\Phi$ when $\Phi$ is not grounded.

\section{Back and forth systems, sharp expansions and their flattenings}

Fix, for the whole of this section, a countable language $L$.  We begin with the well known definition of a back and forth system, but then we add additional adjectives.

\begin{definition}   \label{bf}
Given an arbitrary $L$-structure $M$, a (finitary) {\em back-and-forth system on $M$} is a non-empty set $\F$ of pairs $(\abar,\bbar)$ of finite tuples from $M$ with $\lg(\abar)=\lg(\bbar)$ satisfying:
\begin{enumerate}
\item  For $(\abar,\bbar)\in\F$, the map $\abar\mapsto\bbar$ is q.f.-$L$-elementary (i.e., for any atomic $L$-formula $\alpha(\xbar)$, $M\models\alpha(\abar)\leftrightarrow\alpha(\bbar)$);
and
\item  For all $(\abar,\bbar)\in\F$ and singleton $c\in M$, there is $d\in M$ such that $(\abar c,\bbar d)\in \F$; and dually for every $d'\in M$ there is $c'\in M$ such that $(\abar c',\bbar d')\in\F$.
\end{enumerate}

\end{definition}

For an $L$-structure $M$ of any size, we define the {\em subsequence maps} as follows.
For each $n\in\omega$, all $k\le n$ and all injective $f:k\rightarrow n$, $f$ induces a mapping from $M^n\rightarrow M^k$ defined as:
for every $\abar=(a_0,\dots,a_{n-1})\in M^n$, $\abar\mr{f}$ is the subsequence  $(a_{f(0)},\dots,a_{f(k-1)})$ of $\abar$.
When $k=0$, we identify $M^0$ with the empty sequence $\emptyseq$.  Thus, for every $n\in\omega$, taking $f:0\rightarrow n$ gives
$\abar\mr{f}=\emptyseq$ for every $\abar\in M^n$.
As a special case, when $k=n$, $f\in \Sym(\{0,\dots,n-1\})$ is a permutation and $\abar\mr{f}$ is the associated permutation of $\abar$.

\begin{definition}  \label{sharp} Let $\F$ be any back and forth system on an $L$-structure $M$.
\begin{itemize}
\item We say $\F$ is {\em downward closed} if, for all $k\le n$ and all injective $f:k\rightarrow n$, whenever $(\abar,\bbar)\in\F\cap M^{2n}$, then
$(\abar\mr{f},\bbar\mr{f})\in\F$ as well.  

\item  $\F$ is a {\em sharp} back and forth system on $M$
if $\F$ is both downward closed and, for every $k\in\omega$,  $\F\cap M^{2k}$ is an equivalence relation $E_k$ on $M^k$ for every $k$. [When $k=0$, $E_0$ is the trivial equivalence relation
on $M^0=\{\emptyseq\}$.]
\end{itemize}

\end{definition}  

It is readily verified that  if $\F$ is a back-and-forth system on $M$,  then so is its {\em downward closure} $dc(\F)=\{(\abar\mr{f},\bbar\mr{f}):(\abar,\bbar)\in\F$, $f:k\rightarrow n$ injective$\}$.
As well, if for every $k\in\omega$ we let $E_k$ be the symmetric and transitive closure of $(\dc(\F)\cap M^{2k})\cup \{(\abar,\abar):\abar\in M^k\}$, then
the union of the set of equivalence relations $\{E_k:k\in\omega\}$ is the smallest sharp back-and-forth system  containing $\F$.  
For the remainder of this paper, we will only consider sharp back-and-forth systems.
%

We formalize the operations described by expanding the language.  

\begin{definition}  \label{sharpdef}    For a given countable language $L$, let  $\Lsharp=L\cup\{E_k:k\in\omega\}$, where each $E_k$ is a new $2k$-ary relation symbol.
A {\em sharp structure $(M,\F)$} is an $\Lsharp$ expansion of an $L$-structure $M$ by interpreting each $E_k$ as $\F\cap M^{2k}$ for some distinguished
sharp back-and-forth system $\F$.  The class of sharp structures is (first order) elementary, so fix an axiomatization
$\Axsharp$ of this class.

\end{definition}


Regardless of the size of $M$, there is a unique largest sharp back and forth system on $M$, which we denote by $\F_\infty$
and we let $\Msharp$ denote the canonical sharp expansion $(M,\F_\infty)$.
To describe $\F_\infty$ we recall Karp's notion of back-and-forth equivalence 
(equivalently, $L_{\infty,\omega}$-elementary equivalence) 
of structures.

\begin{definition}  Fix an $L$-structure $M$.  For tuples $\abar,\bbar$ of the same length, we say {\em $(M,\abar)$ and $(M,\bbar)$ are back-and-forth equivalent} if
there is some back-and-system $\F$ containing $(\abar,\bbar)$.  We denote this as $(M,\abar)\equiv_{\infty,\omega}(M,\bbar)$.   This introduces an equivalence relation
on each $n$ and we say $\tp_\infty(\abar)=\tp_\infty(\bbar)$ if $(M,\abar)$ and $(M,\bbar)$ are back-and-forth equivalent.

\end{definition}

Back and forth equivalence is the same as $L_{\infty,\omega}$-elementary equivalence, so $\tp_\infty(\abar)$ is a set of $L_{\infty,\omega}$-formulas.  Moreover,
$\tp_\infty(\abar)$ is isolated by a single $L_{\infty,\omega}$-formula, namely the $L_{\infty,\omega}$-Scott sentence of $(M,\abar)$.
It is straightforward to check that $$\F_\infty:=\{(\abar,\bbar):\tp_\infty(\abar)=\tp_\infty(\bbar)\}$$ is a sharp back-and-forth system, and in fact, $\F_\infty$ contains every
sharp back-and-forth system on $M$.  

Another source of examples is found via expansions.  Suppose $L^*\supseteq L$ and $M^*$ is an expansion of $M$.  
Define $\F^*:=\{(\abar,\bbar):(M^*,\abar)\equiv_{\infty,\omega} (M^*,\bbar)\}$.  
Then $(M,\F^*)$ is also a sharp expansion of $M$ and $\F^*\subseteq\F_\infty$.
In fact, every sharp expansion $(M,\F)$ corresponds to an expansion $M^*$ of $M$, formed by adding a new $k$-ary predicate for every $E_k$-class in $\F$.
Then $\F$ is the largest sharp back and forth system on $M^*$.

\subsection{Flattenings of sharp structures}
Here we introduce the central notion of this paper.

\medskip

Given a sharp structure $(M,\F)$ in the language $\Lsharp=L\cup\{E_n:n\in\omega\}$, it is natural to form a `restricted eq-expansion' by adding new sorts for each of the $2n$-ary equivalence relations
$E_n$.  That is, for each $n$, add a new unary predicate symbol $U_n(z)$ and an $n$-ary function symbol $\pi_n$ and let $(M,\F)^{eq}$ be formed by interpreting each 
predicate  $U_n$ as being a name for the $E_n$-classes, and $\pi_n:M^n\rightarrow U_n$ is the canonical projection, i.e., $\pi_n(\abar)=\abar/E_n$.

Note that the subsequence maps $\abar\mapsto\abar\mr{f}$ are definable in this language, so we also get induced (unary!) projection maps $P^f_{k,n}:U_n\rightarrow U_k$
indexed by injections $f:k\rightarrow n$,
and these commute with the projection maps $\pi_n$, i.e., if $f:k\le n$ is any injection,
then for any $\abar\in M^n$, $P^f_{k,n}\pi_n(\abar)=\pi_k(\abar\mr{f})$.

We now define the {\em flattening of $(M,\F)$}, which naively can be described as the result of passing to the $eq$-structure described above, and then `forgetting the home sort.'
In what follows, $L$-formulas  $\phi(x_0,\dots,x_{n-1})$ (either finitary or infinitary) should be thought of as codes for subsets of $M^n$.  Part of the data of $\phi(x_0,\dots,x_{n-1})$ includes
the $n$, hence for the following definitions, $L$-formulas should not be thought of as being closed under adjunction of dummy variables.  By contrast, in the $\Lflat$ language defined below, each $U_n$ is 
unary, with an element of $U_n^\B$ of a flat structure coding an equivalence class of $n$-tuples from some $L$-structure.  
More formally, 


\begin{definition}   For any language $L$, let $\QF_L=\bigcup\{\QF_n:n\in\omega\}$, where $\QF_n$ is the set of all (finitary) quantifier-free formulas $\alpha_n(x_0,\dots,x_{n-1})$,
such that every free variable is among $\{x_0,\dots,x_{n-1}\}$.  [Every quantifier-free formula $\alpha$ appears in infinitely many $\QF_n$.]
Let 
$$\Lflat\!:=\!\{U_n(z):n\in\omega\}\cup\{\alpha_n^\flat(z):\alpha_n\in\QF_n,n\in\omega\} \cup$$
$$\qquad\cup\{P^f_{k,n}:f:k\rightarrow n\ \hbox{an injection$\}$}\cup\{\const\}$$ where each $U_n$ and $\alpha_n^\flat$ are unary predicates, $P^f_{k,n}:U_n\rightarrow U_k$
are unary function symbols, and $\const$ is a constant symbol.

Given any sharp structure $(M,\F)$, its {\em flattening} is the $\Lflat$-structure $\B=(M,\F)^\flat$ with universe $\bigcup_{n\in\omega} M^n/E_n$, where each 
$U_n^{\B}=M^n/E_n$, $\alpha_n^\flat(\B)=\{\abar/E_n\in U_n:M\models \alpha_n(\abar)\}$, each $P^f_{k,n}:U_n^{\B}\rightarrow U_k^{\B}$ given by $P^f_{k,n}(\abar/E_n)=\abar\mr{f}/E_k$,
and $\const$ is the unique element of $M^0/E_0$.

\end{definition}

In the above, the number of free variables of $\alpha_n$ is critical to its flattening.   In particular, every $\phi_n\in\QF_n$ has
 $\alpha_n^\flat\vdash U_n$, whereas flattenings of elements of $\QF_{n+1}$ entail $U_{n+1}$.  $L$-sentences have flattenings in $U_0$.
 Although being part of the signature, we abuse notation slightly and take the domain of $P^f_{k,n}$ inside any $\Lflat$-structure $\B$ to be $U_n^{\B}$.
In any  flattening, $\alpha_n^\flat(\B)$ and $P^f_{k,n}$ are well defined because $\F$ is a sharp back-and-forth system on $M$.  In particular, if $E_n(\abar,\bbar)$ holds,
then $\qftp_M(\abar)=\qftp_M(\bbar)$.  Note also that  $\const$ is the unique element of $U_0^\B$ and for every quantifier-free $L$-sentence $\sigma$,
$\B\models\sigma_0^\flat(\const)$ if and only if $M\models \sigma$.

The correspondence $\alpha\mapsto\alpha^\flat$ extends naturally to arbitrary $L$-formulas, and even to $L_{\infty,\omega}$-formulas.

\begin{definition} \label{flatformulas}
  For every formula $\phi_n(x_0,\dots,x_{n-1})$ of $L_{\infty,\omega}$, we recursively define an $(\Lflat)_{\infty,\omega}$-formula $\phi^\flat(z)$ as follows:
\begin{itemize}
\item  For $\alpha_n(x_0,\dots,x_{n-1})\in\QF_n$, let $\alpha^\flat_n$ be the distinguished $\Lflat$ predicate symbol;
\item  For formulas  $\{\phi_i(x_0,\dots,x_{n-1}):i\in I\}$, $(\bigwedge_{i\in I}\phi_i)^\flat:=\bigwedge_{i\in I}\phi_i^\flat$;
\item  $(\neg\phi)^\flat:=\neg(\phi^\flat)$; and
\item  $(\exists y\theta(x_0,\dots,x_{n-1},y))^\flat:=\exists w(U_{n+1}(w)\wedge\theta^\flat(w)\wedge P^{id}_{n,n+1}(w)=z)$.
\end{itemize}  
Let $\FM^\flat:=\{\phi^\flat(z):\phi\in L_{\infty,\omega}\}$, which is a subset of $(\Lflat)_{\infty,\omega}$.\newline
For an $\Lflat$-structure $\B$, $n\in\omega$,  and $\aa\in U_n^{\B}$, let
$$\flattp_{\B}(\aa):=\{U_n(z)\}\cup\{\phi^\flat(z)\in\Fm^\flat: \B\models\phi^\flat(\aa)\}$$

\end{definition}

Note that every   $\Lflat$-formula (finitary or infinitary) in $\FM^\flat$ has at most one free variable.  If $\phi^\flat(z)\vdash U_n(z)$ for some $n\ge 1$, it has exactly one free variable.
For $n=0$, since $U_0^\B=\{\const\}$, for any (infinitary) $L$-sentence $\sigma$ we identify the (infinitary) $L^\flat$-sentences 
$\sigma_0^\flat(\const)$ and  $\exists z \sigma_0^\flat(z)$.
For brevity, we write $\sigma^\flat$ in place of $\sigma^\flat(\const)$ for the flattening of  $\sigma$.  
The following Lemma is straightforward.

\begin{lemma}  \label{27} For all $n$ and all $\phi(x_0,\dots,x_{n-1})\in L_{\infty,\omega}$, for all $L$-structures $M$ and for all sharp back-and-forth systems $\F=\{E_n:n\in\omega\}$ on $M$ we have (letting
$\B$ denote the flattening $(M,\F)^\flat$) 
for all $\abar\in M^n$: 
$$M\models \phi(\abar) \quad \hbox{if and only if} \quad \B\models\phi^\flat(\abar/E_n)$$
\end{lemma}

\begin{proof} 
We prove this by induction on the complexity of $\phi$.  For quantifier-free $L$-formulas, this correspondence is built into the definition of a flattening, and it is routine to see this correspondence is preserved under (infinitary) boolean combinations.  To see that it is preserved under quantification, assume the Lemma holds for $\theta(\xbar,y)$ with $\lg(\xbar)=n$.    
Fix any $(M,\F)$
and choose any  $\abar\in M^n$.  Let $\B$ denote the flattening and let $\aa=\abar/E_n$.
On one hand, if $M\models\exists y\theta(\abar,y)$, choose $b\in M$ such that $M\models\theta(\abar,b)$ and let $\bb=\abar b/E_{n+1}\in U_{n+1}$.
As our inductive hypothesis gives $\B\models\theta^\flat(\bb)$, we conclude $\B\models\phi^\flat(\aa)$.
Conversely, choose $\bb\in U_{n+1}$ witnessing $\B\models\phi^\flat(\aa)$, i.e., $\B\models\theta^\flat(\bb)\wedge P^{id}_{n,n+1}(\bb)=\aa$.  
As $\B\models\theta^\flat(\bb)$, by induction there is some $\bbar\in M^n$ and $b^*\in M$
such that $\bbar b^*/E_{n+1}=\bb$ and $M\models\theta(\bbar, b^*)$.  
It follows from the interpretation of $P^{id}_{n,n+1}$ in $\B$ that $\bbar/E_n=\aa$, which is also equal to
$\abar/E_n$.  Thus, $(\abar,\bbar)\in\F$.  As $\F$ is a back-and-forth system, there is $a^*\in M$ such that $(\abar a^*,\bbar b^*)\in\F$, so $\abar a^*/E_{n+1}=\bbar b^*/E_{n+1}=\bb$.  By the inductive hypothesis applied to $\theta(\xbar,y)$, we conclude $M\models\theta(\abar,a^*)$, i.e., $M\models\phi(\abar)$, as required.
\hfill\qedsymbol\end{proof}


\begin{definition}   An $\Lflat$-structure $\B$ is {\em Hausdorff\/} if $\FM^\flat$ separates points, i.e., 
distinct  elements from $\B$ have distinct $\flat$-types.  That is,
 $\flattp_{\B}(\aa)=\flattp_{\B}(\bb)$ implies $\aa=\bb$.

\end{definition}


\begin{lemma} \label{2.9} For any $L$-structure $M$, the flattening of  $\Mflat$, the canonical expansion $\Msharp=(M,\F_\infty)$, is Hausdorff.  Conversely, if $(M,\F)$ is any sharp expansion of $M$
whose flattening is Hausdorff, then $\F=\F_\infty$.
\end{lemma}

\begin{proof} 
Let $\B$ be the flattening of  the canonical $\Msharp$.  To see that $\B$ is Hausdorff,  choose any $\aa\neq\bb$ from $\B$.  If $\aa,\bb$ are in different $U_n$'s they are clearly separated, so assume $\B\models U_n(\aa)\wedge U_n(\bb)$.  By the definition of the flattening, choose $\abar,\bbar\in M^n$ such that $\abar/E_n=\aa$ and $\bbar/E_n=\bb$.  As $\aa\neq\bb$, $(\abar,\bbar)\not\in\F_\infty$,
so $\tp_\infty(\abar)\neq\tp_\infty(\bbar)$.  Choose any infinitary $\phi(\xbar)\in\tp_\infty(\abar)\setminus\tp_\infty(\bbar)$.  Then $\B\models\phi^\flat(\aa)\wedge\neg\phi^\flat(\bb)$,
so $\B$ is Hausdorff.  

For the converse, recall that for any back-and-forth system $\F$ on $M$, if $(\abar,\bbar)\in\F$, then $\tp_{\infty}(\abar)=\tp_\infty(\bbar)$, hence $\F\subseteq\F_\infty$.
Now assume $\F$ is any sharp back-and-forth system that is not equal to $\F_\infty$, and we show that its flattening $\B$ cannot be Hausdorff.
Since $\F_\infty\not\subseteq\F$, there is some  $n$ and $\abar,\bbar\in M^n$ with $(\abar,\bbar)\in\F_\infty\setminus\F$, i.e., $\tp_\infty(\abar)=\tp_\infty(\bbar)$, but
$\aa\neq\bb$, where $E_n=\F\cap M^{2n}$, $\aa=\abar/E_n$ and $\bb=\bbar/E_n$.  
Thus, $\aa,\bb\in\B$ and for every $\phi(\xbar)\in L_{\infty,\omega}$, $\B\models\phi^\flat(\aa)\leftrightarrow\phi^\flat(\bb)$.
So $\aa$ and $\bb$ witness that $\B$ is not Hausdorff.
\hfill\qedsymbol
\end{proof}

\section{Flat structures, the inverse map, and potential canonical Scott sentences}  \label{Hdfstuff}

\subsection{Axiomatizing flat structures}

We formalize the notion of a flat structure $\B$.  We will see that every flattening of a sharp expansion is a flat structure, and conversely, every {\em countable} $\B$
is the flattening of some sharp expansion of some countable $L$-structure.  
  Throughout, we have a fixed (countable) language $L$, along with the associated language $\Lflat$.
  
  As notation, for any $\Lflat$-structure $\B$ and any $\bb\in U_n$, let $$\Delta_{\bb}:=\{\hbox{quantifier-free}\ \beta(x_0,\dots,x_{n-1}):\B\models\beta^\flat(\bb)\}$$
  In particular, $\Delta_\bb$ is a set of  finitary $L$-formulas, all of whose free variables are among $\{x_0,\dots,x_{n-1}\}$.
  The definition below lists several conditions that hold of any flattening of a sharp structure $(M,\F)$.

\begin{definition} \label{flatdef}  Let $\Ax$ denote the following  $L^\flat$ axioms,\footnote{All but 1a) are first order axioms.} and call an $\Lflat$-structure $\B$ a {\em flat structure} if
$\B\models\Ax$.
\begin{enumerate} 
\item  Structural axioms:  
\begin{enumerate}
\item  The unary predicates $\{U_n:n\in\omega\}$ partition the universe.
\item  $U_0=\{\const\}$ is a singleton.
\item  For all  $k\le n$ and injective $f:k\rightarrow n$, $P^f_{k,n}$ is a (unary) function$:U_n\rightarrow U_k$.
\item  For all $n\in\omega$, $P^{id}_{n,n}:U_n\rightarrow U_n$ (where $id:n\rightarrow n$ is the identity) is the identity function.
\item  For all $n\in\omega$ and all $\bb\in U_n$, $\Delta_\bb$ is a complete quantifier-free type in the variables $(x_0,\dots,x_{n-1})$ (i.e.,
there is some $L$-structure $M$ and some $\abar\in M^n$ such that $M\models\beta(\abar)$ for all $\beta\in \Delta_\bb$ and, for every 
quantifier-free $\beta(x_0,\dots,x_{n-1})$, exactly one of $\beta,\neg\beta$ is contained in $\Delta_\bb$.)
\end{enumerate}
\item  Relational axioms:
\begin{enumerate}
\item  Composition: For all injective $f:k\rightarrow n$ and $g:n\rightarrow m$, and all $\aa\in U_m$, $P^{g\circ f}(\aa)=P^f_{k,n}(P^g_{n,m}(\aa))$.
\item  For all injective $f:k\rightarrow n$, all $\aa\in U_n$,  and all quantifier-free $\alpha_k(x_0,\dots,x_{k-1})$,
$$\alpha_k(x_0,\dots,x_{k-1})\in\Delta_{P^f_{k,n}(\aa)} \quad \hbox{if and only if} \quad \alpha_k(x_{f(0)},\dots,x_{f(k-1)})\in\Delta_\aa$$
\end{enumerate}
\item Equality Axiom:  For all $k\le n$, all injective $f,g:k\rightarrow n$ and all $\aa\in U_n$,
if $(\bigwedge_{i\in k} x_{f(i)}=x_{g(i)})\in \Delta_\aa$, then $P^f_{k,n}(\aa)=P^g_{k,n}(\aa)$.
\end{enumerate}

 {\bf As notation:} If $\aa\in U_k$ and $\bb\in U_n$,  $\aa\le\bb$ denotes $\aa=P^{id}_{k,n}(\bb)$
 (so $k\le n$).

\begin{enumerate} \setcounter{enumi}{3}
\item  Amalgamation:  For all $k\le n,m$, if $\aa\in U_k$, $\bb\in U_n$, $\cc\in U_m$, $\aa\le\bb$  and $\aa\le\cc$, then there is $\dd\in U_{n+m-k}$ such that
$\bb\le \dd$ and $P^v(\dd)=\cc$, where $v:m\rightarrow n+(m-k)$ is defined by
$v(i)=i$ for all $i<k$ and $v(i)=n+(i-k)$ for all
$k\le i<m$.


\item  Duplication:  $\B\models\sigma^\flat$, where $\sigma$ is the $L$-sentence $\forall x_0\exists x_1 (x_0=x_1)$, i.e.,
for every  $\aa\in U_1$, there is $\bb\in U_2$ such that $P^{f_0}(\bb)=P^{f_1}(\bb)=\aa$ for the two functions $f_0,f_1:1\rightarrow 2$ 
and $(x_0=x_1)\in\Delta_\bb$. 
 
\item  Constants:  For all constant symbols $c\in L$, $\B\models \sigma_c^\flat$, where $\sigma_c$ is the $L$-sentence $\exists x_0(x_0=c)$.
\item  Functions:  For all $k$-ary function symbols $f\in L$, $\B\models\sigma_f^\flat$, where $\sigma_f$ is the $L$-sentence
$\forall x_0\dots\forall x_{k-1}\exists x_k(x_k=f(x_0,\dots,x_{k-1}))$.
\end{enumerate}

\end{definition}

%
%
We remark that without additional conditions, the element $\dd$ obtained by Amalgamation is not uniquely described.


\subsection{Generalized subsequences}  \label{Generalized}


In preparation of Proposition~\ref{reverse}, where we recover an $L$-structure $M$ from a flat structure $\B$, 
we must show that elements of $\B$ encode tuples from $M$ 
with duplicate values.
We begin with a 
one-step lemma, whose proof follows immediately from Amalgamation, Duplication and Axiom 2b.

\begin{lemma} \label{onestep}   Let $\B$ be any flat structure.
Suppose $m\ge 1$, $j\in m$, and $\aa\in U_m$.  There is $\cc\in U_{m+1}$ such that $\aa\le\cc$ and $(x_j=x_m)\in \Delta_{\cc}$.
\end{lemma}

By iterating Lemma~\ref{onestep} $n$ times, we get the following lemma.

\begin{lemma}  \label{generalizedextension}  Suppose $f^*:n\rightarrow m$ is any function (not necessarily injective).
For any flat structure $\B$ and any $\aa\in U_m$,
there is a unique  $\cc\in U_{m+n}$ such that $\aa\le\cc$ and $(\bigwedge_{i\in n} x_{f^*(i)}=x_{m+i})\in\Delta_{\cc}$.
\end{lemma}

\begin{proof}  The existence of such a $\cc$ follows by iterating Lemma~\ref{onestep}  $n$ times.  As for uniqueness, suppose $\cc_1,\cc_2\in U_{m+n}$ both satisfy these conditions.
Let $\dd\in U_{m+2n}$ be an amalgamation of $\cc_1$ and $\cc_2$ over $\aa$.  
In particular,  $\cc_1\le \dd$, and $\cc_2=P^g_{m+n,m+2n}$, where $g$ is defined as
$g(i)=i$ for all $i<m$ and $g(i)=i+n$ if $m\le i<m+n$.
  Because of the Equality Axiom, to conclude that $\cc_1=\cc_2$, it suffices to show 
that $(x_{id(i)}=x_{g(i)})\in \Delta_\dd$ for all $i<m+n$.  
We verify this by cases.  
For $i<m$,  since $\aa\le\cc_2$ we have $g(i)=i$, hence $(x_i=x_{g(i)})\in\Delta_\dd$ trivially.
For
$m\le i<m+n$, let $j=i-m$.  
By hypothesis, since $i=m+j$,  $(x_{f^*(j)}=x_{i})$ is in both $\Delta_{\cc_1}$ and $\Delta_{\cc_2}$.  
Since $\cc_1\le\dd$, we have $(x_{f^*(j)}=x_{i})\in\Delta_{\dd}$ as well.  But $P^g_{m+n,m+2n}(\dd)=\cc_2$ implies
$(x_{g(f^*(j))}=x_{g(i)})\in\Delta_\dd$.  Because $f^*(j)<m$, $g(f^*(j))=f^*(j)$, so $(x_{f^*(j)}=x_{g(i)})\in\Delta_{\dd}$.
As $\Delta_\dd$ is a complete q.f.\ type, we conclude $(x_{i}=x_{g(i)})\in\Delta_\dd$.
Thus, as noted above, we conclude $\cc_1=\cc_2$ by the Equality Axiom.
\hfill\qedsymbol\end{proof}

In light of this uniqueness result, we make the following definitions.  For the remainder of this subsection, we assume $\B$ is a fixed flat structure, and we work inside $\B$.

\begin{definition} Suppose $f^*:n\rightarrow m$ is any function and $\aa\in U_m$. 
 
The {\em $f^*$-blowup of $\aa$} is the (unique) $\cc\in U_{m+n}$ satisfying
$\aa\le\cc$ and $(\bigwedge_{i\in n} x_{f^*(i)}=x_{m+i})\in\cc$.   

For any integer $\ell$, the (injective) {\em $\ell$-shift function $v_\ell:n\rightarrow n+\ell$} is defined by $v_\ell(i)=\ell+i$.

The {\em generalized projection} $P^{f^*}(\aa)$ is the element $\bb\in U_n$ satisfying
$\bb= P^{v_m}_{n,m+n}(\cc)$, where $\cc$ is the $f^*$-blowup of $\aa$ and $v_m:n\rightarrow m+n$ is the $m$-shift function.

\end{definition}

We show that these generalized projections satisfy statements analogous to the axioms for (hardwired) projections $P^f$ (for injective functions $f$).
We begin with a Generalized Equality criterion.

\begin{lemma}  Suppose $f^*,g^*:n\rightarrow m$ and $\aa\in U_m$ satisfies $(\bigwedge_{i\in n} x_{f^*(i)}=x_{g^*(i)})\in\Delta_\aa$.  Then $P^{f^*}(\aa)=P^{g^*}(\aa)$.
\end{lemma}

\begin{proof} Let $\cc\ge\aa$ denote the $f^*$-blowup of $\aa$.  As $f^*(i)=g^*(i)$ for each $i\in n$, it follows from Lemma~\ref{generalizedextension} that $\cc$ is also the $g^*$-blowup of
$\aa$.  Thus, by definition, $P^{f^*}(\aa)=P^v_{n,m+n}(\cc)=P^{g^*}(\aa)$.
\hfill\qedsymbol\end{proof}

To show that Generalized Composition holds, we introduce some more notation.   Given a sequence $\<\aa_0,\dots,\aa_{\ell-1}\>\in\B^\ell$ with $\aa_i\in U_{k_i}$,
let $s_i=\sum_{j\le i} k_j$ and let $s=s_{\ell-1}$.  For each $i < \ell$, let $f_i: k_i \rightarrow s$ be the $s_{i-1}$-shift function defined by $f_i(j) = s_{i-1}+j$ (where we take $s_{-1} = 0$). 
A {\em joint embedding of  $\<\aa_0,\dots,\aa_{\ell-1}\>$} is any $\dd\in U_s$ such that $P^{f_i}(\dd)=\aa_i$ for every $i\le\ell$. $\dd$ is not uniquely determined. 

It is easily checked that for any sequence $\<\aa_0,\dots,\aa_{\ell-1}\>$ and any joint embedding $\dd\in U_s$, 
$\aa_0\le\dd$ and, letting $v_{k_0}:(s-k_0)\rightarrow s$ be the $k_0$-shift function, the 	`truncation' $P^{v_{k_0}}(\dd)$ is a joint embedding of $\<\aa_1,\dots,\aa_{\ell-1}\>$.
Given this, Generalized Composition follows easily.

\begin{lemma}  Suppose $f^*:k\rightarrow n$ and $g^*:n\rightarrow m$ are arbitrary functions.  For any $\aa\in U_m$  we have
$P^{g^*f^*}(\aa)=P^{f^*}(P^{g^*}(\aa))$.
\end{lemma}

\begin{proof} Let $\bb\in U_{m+n}$ be the $g^*$-blowup of $\aa$, let $\cc\in U_{m+k}$ be the $g^*f^*$-blowup of $\aa$, and, as $\aa\le\bb$ and $\aa\le\cc$, let
$\dd$ be an  amalgamation of $\bb$ and $\cc$ over $\aa$.  [In fact, $\dd$ is uniquely determined, but we don't need this.]
Then $\dd$ is a joint embedding of $\<\aa,P^{g^*}(\aa),P^{g^*f^*}(\aa)\>$. 
 Let $\ee=P^{v_m}(\dd)$ be the truncation of $\dd$; so by the preceding remarks, $\ee$ is a joint embedding of $\<P^{g^*}(\aa),P^{g^*f^*}(\aa)\>$.  We claim that $\ee$ is the $f^*$-blowup of
$P^{g^*}(\aa)$.  This observation completes the proof, since the $f^*$-blowup of $P^{g^*}(\aa)$ is a joint embedding of $\<P^{g^*}(\aa),P^{f^*}(P^{g^*}(\aa))\>$; and hence, by comparing truncations, we get $P^{g^*f^*}(\aa)=P^{f^*}(P^{g^*}(\aa))$.

To prove the claim, since $\ee$ is a joint embedding,  $P^{g^*}(\aa)\le \ee$.  It remains to show that $(x_{f^*(i)}=x_{n+i})\in \Delta_\ee$ for all $i<k$.
Fix $i<k$ and let $j=f^*(i)\in n$.   Now, in $\bb$, $(x_{g^*(j)}=x_{m+j})\in\Delta_\bb$, hence in $\Delta_\dd$, since $\bb\le\dd$.
On the other hand, in $\cc$, $(x_{g^*f^*(i)}=x_{m+i})\in\Delta_\cc$, hence $(x_{g^*f^*(i)}=x_{m+n+i})\in\Delta_\dd$ from the amalgamation.  As $\Delta_\dd$ is a complete type, these two statements imply $(x_{m+j}=x_{m+n+i})\in\Delta_\dd$.
As $\ee=P^{v_m}(\dd)$ and recalling $j=f^*(i)$, we have $(x_{f^*(i)}=x_{n+i})\in\Delta_\ee$, as required.
\hfill\qedsymbol\end{proof}

\subsection{Reconstructing $(M,\F)$ from a countable $\B$}

Throughout this section we fix a countable language $L$ and a countable, flat structure $\B$ in the language $\Lflat$.  Armed with our results from Subsection~\ref{Generalized}, our index
functions $f:k\rightarrow m$ are not assumed to be injective, but to ease notation we write $f$ in place of $f^*$.
The goal of this subsection is to prove the following proposition.

\begin{proposition} \label{reverse} For every countable flat structure $\B$, there is a countable $L$-structure $M$ and a sharp back-and-forth system $\F$ such that
$(M,\F)^\flat\cong \B$.  The sharp structure $(M,\F)$ is unique up to $L^\sharp$-isomorphism.  
Moreover, 
any $\Lflat$-isomorphism $\sigma^\flat:\B_1\rightarrow \B_2$ of countable flat structures
lifts to an $\Lsharp$-isomorphism $\Phi^*:(M_1,\F_1)\rightarrow (M_2,\F_2)$ satisfying $\sigma^\flat(\pi^1_n(\abar))=\pi^2_n(\Phi^*(\abar))$ for all $n$ and all $\abar\in M_1^{n}$.
\end{proposition}

Towards the proof of Proposition~\ref{reverse}, we introduce a new notion.

\begin{definition}   A {\em cofinal sequence in $\B$} $\<\aa_n:n\in\omega\>$ satisfies: For all $n\in\omega$
\begin{enumerate}
\item  $\aa_n\in U_n$ and $\aa_n\le\aa_{n+1}$; and
\item  For all $n,k,\ell\in\omega$, for all $\bb\in U_k$ and $\cc\in U_\ell$, if $\bb=P^f_{k,n}(\aa_n)$ and $\bb=P^g_{k,\ell}(\cc)$,
then there is an $m\ge n$ and a function $g':\ell\rightarrow m$ such that $g'\circ g=id\circ f$ and $P^{g'}(\aa_m)=\cc$.
\end{enumerate}

\end{definition}  

That is, a cofinal sequence represents a kind of Fra\"iss\'e sequence of elements from $\B$.  The following lemma is easily achieved by proper bookkeeping.

\begin{lemma}  Every countable flat structure $\B$ has a cofinal sequence.  In fact, any finite sequence $\<\aa_i:i<n\>$ satisfying $\aa_i\in U_i$ and $\aa_i\le\aa_{i+1}$ for each $i$ can be extended to a cofinal sequence.
\end{lemma}

\begin{Construction} \label{construction}
Let $\B$ be a countable flat structure  and fix a cofinal sequence $\<\aa_n:n\in\omega\>$.  We construct a countable $L$-structure $M$ and a sharp back-and-forth system $\F=\{E_n:n\in\omega\}$ such that $(M,\F)^\flat$ is $\Lflat$-isomorphic to $\B$ as follows:
\end{Construction}

\begin{itemize}
\item  First, as $\aa_n\le\aa_{n+1}$, we have $\Delta_{\aa_n}\subseteq\Delta_{\aa_{n+1}}$, and as each $\Delta_{\aa_n}$ describes a complete, quantifier-free type in the variables
$(x_0,\dots,x_{n-1})$, the union $q(x_0,x_1,\dots):=\bigcup_{n\in\omega} \Delta_{\aa_n}$ describes a complete, quantifier-free type in $\omega$ variables $X=\{x_i:i\in\omega\}$.

\item Next, define an equivalence relation $\sim$ on $X$ as:  $x_i\sim x_j$ iff $(x_i=x_j)\in q$.  As $q$ is a complete type, $\sim$ indeed is an equivalence relation.

\item  Define an $L$-structure $M$ to have universe $X/{\sim}$.  For any quantifier-free $L$-formula $\alpha$ 
with $n$ free variables and any $(i_0,\dots,i_{n-1})\in\omega^n$, say $M\models\alpha([x_{i_0}],\dots, [x_{i_{n-1}}])$ if and only if $\alpha(x_{i_0},\dots,x_{i_{n-1}})\in q$.
It is readily checked that this is well defined and defines an $L$-structure.

\item  Additionally, for each $n\in \omega$, define a {\em covering map} $cov_n:M^n\rightarrow U_n$ as:  For any $\abar=(a_0,\dots,a_{n-1})\in M^n$ (so each $a_j$ is a $\sim$-class of elements from $X$), let $cov_n(\abar)=P^h(\aa_m)$ for any/all sufficiently large $m$ and any function $h:n\rightarrow m$ satisfying $x_{h(j)}\in a_j$ for all $j<n$.
Again, it is easily checked that this definition does not depend on our choice of $m$ or $h$.

\item  Finally, for each $n\in\omega$, define $E_n(\xbar,\ybar)$ on $M^{2n}$ as $E_n(\abar,\bbar)\Leftrightarrow cov_n(\abar)=cov_n(\bbar)$ and let $\F:=\bigcup_{n\in\omega} E_n$.
\end{itemize}

\begin{remark}  \label{cf} 
It is easily verified that if $(M,\F)$ is a countable, sharp structure and $\B=(M,\F)^\flat$ is its flattening, then for any enumeration $\<a_i:i\in\omega\>$ of $M$, Construction~\ref{construction} is an inverse
of the flattening map.  That is, taking $\aa_n:=\pi_n(\<a_i:i<n\>)$, 
the sequence $\<\aa_n:n\in\omega\>$ is cofinal, and taking   $cov_n:=\pi_n$ for each $n$, the sharp structure  we obtain from Construction~\ref{construction} is $\Lsharp$-isomorphic to $(M,\F)$.

\end{remark}

Much of the verification that  $\F$ is a sharp back-and-forth system and that $(M,\F)^\flat\cong \B$ 
 is codified with the following fundamental lemma, which will be used in later sections as well.

\begin{lemma}  \label{stepup}  Suppose $n\in\omega$, $\abar\in M^n$, $cov_n(\abar)=\bb\in U_n$, and $\cc\in U_{n+k}$ with $\bb\le\cc$.
Then there is a $k$-tuple $\cbar\in M^k$ such  that $cov_{n+k}(\abar\cbar)=\cc$.
\end{lemma}

\begin{proof}  Arguing by induction on $k$, it suffices to prove this when $k=1$.
Say $\abar=([x_{i_0}],\dots,[x_{i_{n-1}}])$ and choose any $m>\max\{i_0,\dots,i_{n-1}\}$.  Let $f:n\rightarrow m$ be given by $f(j)=i_j$.
By the definition of $cov_n$, $P^f_{n,m}(\aa_m)=\bb$.  As $\<\aa_k:k\in\omega\>$ is cofinal, there is some $\ell\ge m$ and $h:n+1\rightarrow\ell$
such that $id\circ f=h\circ id$ and $P^h(\aa_\ell)=\cc$.  The commutativity of the functions gives $h(j)=f(j)$ for all $j<n$, so taking $c=[x_{h(n)}]$
we have $cov_{n+1}(\abar c)=\cc$, as required.
\hfill\qedsymbol\end{proof}

Using Lemma~\ref{stepup}, we now verify that  the statements of Proposition~\ref{reverse}.
First,  it is easily checked that $\F$ is a sharp back-and-forth system on $M$.
Second, taking $n=0$ and $k$ arbitrary, we see that $cov_k:M^k\rightarrow U_k$ is onto, so $\B$ is isomorphic to the flattening of $(M,\F)$, with $cov_n$
playing the role of $\pi_n$.

The uniqueness statement follows from the Moreover clause, which we now establish. 
Fix an $\Lflat$-isomorphism $\sigma^\flat:\B_1\rightarrow\B_2$ between countable flat structures. 
Choose any countable sharp structures $(M_1,\F_1)$, $(\M_2,\F_2)$ whose flattenings are $\B_1,\B_2$, respectively.
For each $\ell=1,2$, by possibly replacing $(M_\ell,\F_\ell)$ by an isomorphic copy, by Remark~\ref{cf} we may assume it is obtained from $\B_\ell$ via Construction~\ref{construction}, 
where $cov^\ell_n=\pi^\ell_n$ for each $n$.
To build an $\Lsharp$-isomorphism,  consider the set of all partial bijections $\Phi:\abar\rightarrow\bbar$, where $\abar\in M_1^n$, $\bbar\in M_2^n$, and 
$\sigma^\flat(\pi_n^1(\abar))=\pi_n^2(\bbar)$.
Clearly, every such $\Phi$ preserves quantifier-free types, as, letting $\cc=\pi_n^1(\abar)$, the $L$-quantifier-free types $\qftp_{M_1}(\abar)$
and $\qftp_{M_2}(\bbar)$ can be read off from $\Delta_{\cc}$ and $\Delta_{\sigma^\flat(\cc)}$.

By Lemma~\ref{stepup}, this family of partial bijections $\Phi$ is a back and forth system.  As both $M_1$ and $M_2$ are countable, it follows that
there is an $L$-isomorphism $\Phi^*:M_1\rightarrow M_2$ that satisfies $\sigma^\flat(\pi_n^1(\abar))=\pi_n^2(\Phi^*(\abar))$ for every $\abar\in M_1^n$.  But this condition implies
that $\Phi^*$ is also $E_n$-preserving:  If $E_n(\abar,\abar')$ in $M_1$, then by definition of $E_n$, we have $\pi_n^1(\abar)=\pi_n^1(\abar')$.  From the sentence above,
$\pi_n^2(\Phi^*(\abar))=\pi_n^2(\Phi^*(\abar'))$, so $E_n(\Phi^*(\abar),\Phi^*(\abar'))$ holds in $M_2$ as well.  Thus, $\Phi^*$ is an $\Lsharp$-isomorphism from
$(M_1,\F_1)$ onto $(M_2,\F_2)$.

%

%
%


\subsection{Hausdorff flat structures are potential canonical Scott sentences}

Fix a countable language $L$.  The aim of this section is to distinguish a canonical class $\Hdf^*$ of Hausdorff flat structures in $\V$.
We will  see that
 every Hausdorff  flat structure is isomorphic  [equivalently, back-and-forth equivalent by Lemma~\ref{flatequiv}] to exactly one canonical structure, so we can naturally identify $\Hdf^*$ with 
 isomorphism classes [equivalently, back and forth classes] of Hausdorff flat structures. Further, we will show that $\Hdf^*$ is naturally bijective with $\CSS(L)_{\ptl}$.



As a preamble, we describe
 some expansions of $L$.  For each $k\in\omega$, let $L(\cbar_k):=L\cup\{c_0,\dots,c_{k-1}\}$ (so $L_0=L$) where $\{c_0,\dots,c_{k-1}\}$ are new constants.
For any $L(\cbar_k)$-structure $N=(M,\abar)\in HC$, let $\css_k(M,\abar)$ denote the canonical Scott sentence of $N=(M,\abar)$.
For  any sentence $\sigma\in (L(\cbar_k))_{\infty,\omega}$, let $\hat{\sigma}(\xbar)$ be the $L_{\infty,\omega}$ formula formed by replacing each $c_i$ by $x_i$.
Note that $\hat{\sigma}(\cbar_k)$ is persistently equivalent to $\sigma$, i.e.,
for any forcing extension $\V[G]\supseteq
\V$ and any $L(\cbar_k)^{\V[G]}$-structure $N=(M,\abar)$, $N\models\sigma$ if and only if $M\models\hat{\sigma}(\abar)$.

\begin{definition}  \label{3.13}  For any language $L$, call two $L_{\infty,\omega}$-sentences $\sigma$ and $\psi$ {\em explicitly contradictory} if $\sigma\wedge\tau$ has no models in any
forcing extension $\V[G]\supseteq\V$.  An $L_{\infty,\omega}$-sentence $\delta$ is {\em persistently valid} if, for every forcing extension $\V[G]\supseteq\V$,
every $M\in\Mod_{\V[G]}(L)$ models $\delta$.

\end{definition}

As examples, note that for any $k\in\omega$, it follows from Fact~\ref{later} that any two canonical Scott sentences $\css_k(M,\abar)$ and $\css_k(M',\abar')$ in 
$(L(\cbar_k))_{\omega_1,\omega}$ are explicitly contradictory or equal.  
Note that $\sigma$ and $\tau$ are explicitly contradictory  $(L(\cbar_k))_{\infty,\omega}$-sentences if and only if  the $L_{\infty,\omega}$-sentence
$\forall\xbar (\hat{\sigma}(\xbar)\rightarrow\neg\hat{\tau}(\xbar))$ is persistently valid.

\bigskip


We begin in the countable world, i.e., in HC.  Recalling Definition~\ref{flatdef}, let  $\Flat_{HC}(L):=\Mod_{HC}(\Ax)$, the flat $\Lflat$-structures in $HC$.
Note that as satisfaction is absolute between transitive models of set theory, the classes $\Flat_{HC}(L)$ and $\Mod_{HC}(L)$ are
strongly definable, i.e., are defined by $HC$-invariant formulas.

\begin{definition}   Let $\can:\Mod_{HC}(L)\rightarrow \Flat_{HC}(L)$ be the following canonical map.
Given $M\in \Mod_{HC}(L)$, $\can(M)$ is the  flat structure  with universe
$$\{\css_k(M,\abar):k\in\omega,\abar\in M^{k}\}$$
where the $\Lflat$ symbols are interpreted as: $U_k(\can(M))=\{\css_k(M,\abar):\abar\in M^k\}$; 
$\can(M)\models \alpha^\flat(\css_k(M,\abar))$ if and only if
$M\models\alpha(\abar)$
for each quantifier-free $L$-formula $\alpha(x_0,\dots,x_{k-1})$;
and for an injective $f:k\rightarrow n$, $P^f_{k,n}(\css_n(M,\abar))=\css_k(M,\abar\mr{f})$.

\end{definition}

  For any
$M\in\Mod_{HC}(L)$, the structure $\can(M)$ is fundamentally an unpacking of the Scott sentence of $M$, 
and hence only depends on the isomorphism type of $M$.
As well, the mappings $\css_n(\abar)\mapsto \pi_n(\abar)$ give an $\Lflat$-isomorphism between
$\can(M)$ and the  flattening $M^\flat=(M,\F_\infty)^\flat$.  In more detail:

\begin{lemma}  \label{samemodel}  If $M,N\in\Mod_{HC}(L)$, then $M\cong N$ if and only if $\can(M)=\can(N)$.
\end{lemma}

\begin{proof}  
Right to left is easy, as taking $k=0$, $U_0(\can(M))=U_0(\can(N))=\{\css(M)\}$.  
Thus, $M\cong N$ by the salient property of canonical Scott sentences.
Now assume $h:M\rightarrow N$ is an $L$-isomorphism.  Then, for any $k\in\omega$ and any $\abar\in M^k$, $(M,\abar)\cong (N,h(\abar))$ as $L(\cbar_k)$-structures,
hence $\css_k(M,\abar)=\css_k(N,h(\abar))$.  Thus, for each $k$ we have
$$\{\css_k(M,\abar):\abar\in M^k\}=\{\css_k(N,\bbar):\bbar\in N^k\}$$
so the universes of $\can(M)$ and $\can(N)$ are equal.  Similarly, it is easily checked that the $\Lflat$-interpretations in  $\can(M)$ and $\can(N)$ are equal as well.
Thus, $\can(M)=\can(N)$.
\hfill\qedsymbol\end{proof}

Using Lemma~\ref{samemodel} and the uniqueness statement of Proposition~\ref{reverse}, the following map is well-defined.

\begin{definition} Let $\bcan:\Flat_{HC}(L)\rightarrow\Flat_{HC}(L)$ be defined as:\\ $\bcan(\B):=\can(M)$, where $M$ is any $L$-structure obtained from
Proposition~\ref{reverse}, i.e., where $(M,\F)^\flat\cong \B$ for some $\F$.  

Let
$\Hdf^*_{HC}$ denote the image of $\bcan$, which is also the image of $\can$.

\end{definition}

This  notation is justified by the following lemma.

\begin{lemma}  \label{canfacts}
\begin{enumerate}
\item  Every $\B'\in\Hdf_{HC}^*$ is Hausdorff.
\item  For all  Hausdorff $\B_1,\B_2\in\Flat_{HC}(L)$, $\B_1\cong\B_2$ if and only if $\bcan(\B_1)=\bcan(\B_2)$.
\item  $\bcan:\Flat_{HC}(L)\rightarrow\Flat_{HC}(L)$ is a projection onto $\Hdf_{HC}^*$, i.e., $\bcan(\bcan(\B))=\bcan(\B)$ for all $\B\in\Flat_{HC}(L)$.
\item  $\Hdf^*_{HC}=\{\B\in\Flat_{HC}(L):\bcan(\B)=\B\}$.
\end{enumerate}
\end{lemma}

\begin{proof} (1)  Fix distinct $\aa\neq\bb\in\B'$.   For definiteness, assume $\aa,\bb\in U_k(\B')$.
 Choose any $M\in\Mod_{HC}(L)$ with $\can(M)=\B'$ and suppose $\aa=\css_k(M,\abar)$ and $\bb=\css_k(M,\bbar)$.
 Let $\sigma:=\css_k(M,\abar)$ and $\tau:=\css_k(M,\bbar)$.  As these are distinct canonical Scott sentences, $\sigma$ and $\tau$ are explicitly contradictory.
 Thus, $M\models\forall\xbar (\hat{\sigma}(\xbar)\rightarrow\neg\hat{\tau}(\xbar))$.  From this, we have $M\models\hat{\tau}(\bbar)\wedge\neg\hat{\tau}(\abar)$.
 As $\B'\cong M^\flat$,  $\B'\models\hat{\tau}^\flat(\bb)\wedge\neg\hat{\tau}^\flat(\aa)$.  Thus, $\B'$ is Hausdorff.

(2)  
Since for $\ell=1,2$, each $\B_\ell$ is Hausdorff, by Lemma~\ref{2.9} and Proposition~\ref{reverse} we can choose  $L$-structures $M_\ell$ such that $M_\ell^\flat\cong\B_\ell$.
First, assume $\bcan(\B_1)=\bcan(\B_2)$.  Then $\can(M_1)=\can(M_2)$, so $M_1\cong M_2$ by Lemma~\ref{samemodel}.  It follows by the uniqueness of flattenings
that
$\B_1\cong\B_2$ as $L^\flat$-structures.
Conversely, assume $\B_1\cong\B_2$.  Then, by the uniqueness sentence of Proposition~\ref{reverse}, $M_1\cong M_2$,
so again by Lemma~\ref{samemodel}, $\can(M_1)=\can(M_2)$. Thus, $\bcan(\B_1)=\bcan(\B_2)$.

(3)  Given $\B\in \Flat_{HC}(L)$, by Proposition~\ref{reverse} choose a sharp structure $(M,\F)$ such that $(M,\F)^\flat\cong \B$.
Put $\B':=\bcan(\B)$, the latter being  equal to $\can(M)$.  
However, since $\B'\cong(M,\F_\infty)^\flat$, $\bcan(\B')=\can(M)$ follows immediately.
Thus, $\bcan(\B')=\can(M)=\B'$, as required.

(4)  Choose any $\B\in\Hdf^*_{HC}$.  That $\bcan(\B)=\B$ is given by (3).  Conversely, assume $\bcan(\B)=\B$.  Then trivially, $\B\in\Hdf^*_{HC}$ by its definition.
\hfill\qedsymbol\end{proof}

So far, in the HC world, the surjective maps $\can:\Mod_{HC}(L)\rightarrow\Hdf^*_{HC}$ and $\bcan:\Flat_{HC}(L)\rightarrow\Hdf^*_{HC}$ are somewhat dual.  However, this similarity will
dissipate when we take potentials.  This is due not only to the fact that $\bcan$ is a projection, but also because there is a preferred homomorphism $j:\B\rightarrow\bcan(\B)$
that we now describe.

\begin{definition}  Suppose $\B,\B'\in \Flat_{HC}(L)$.  A {\em surjective homomorphism $j:\B\rightarrow\B'$} is an onto map such that: (1) for all $\aa\in \B$,
$\B\models U_k(\aa)$ iff $\B'\models U_k(j(\aa))$; (2) for all quantifier-free $L$-formula $\alpha(\xbar)$
$\B\models\alpha^\flat(\aa)$ iff $\B'\models\alpha_k^\flat(j(\aa))$; and (3) $j$ commutes with each $P^f_{k,n}$.

\end{definition}

  Recall that by Definition~\ref{flatformulas}, for any flat structure $\B$ and any $\aa\in\B$, $\flattp_{\B}(\aa)=\{\phi^\flat(z):\B\models\phi^\flat(\aa)\}$.
The following Lemma is proved by  induction on the complexity of $L_{\infty,\omega}$-formulas.

\begin{lemma}  \label{keyH}  Let $j:\B\rightarrow\B'$ be any surjective homomorphism of flat structures.  Then, for any $\aa\in\B$,
$\flattp_{\B}(\aa)=\flattp_{\B'}(j(\aa))$.
\end{lemma}

\begin{lemma}  \label{uniquemap}  For any $\B,\B'\in \Flat_{HC}(L)$ with $\B'$ Hausdorff, there is at most one
surjective homomorphism $j:\B\rightarrow\B'$.  
In particular, for any $\B\in \Flat_{HC}(L)$, there is a unique surjective homomorphism $j:\B\rightarrow\bcan(\B)$.
\end{lemma}

\begin{proof}  Choose surjective homomorphisms $j,j^*:\B\rightarrow\B'$ and choose any $\aa\in\B$.  Let $\bb:=j(\aa)$ and $\bb^*:=j^*(\aa)$.
To show that $\bb=\bb^*$, by applying Lemma~\ref{keyH} twice, we have $\flattp_{\B'}(\bb)=\flattp_{\B'}(\bb^*)$.  Thus, $\bb=\bb^*$ since $\B'$ is Hausdorff.

For the second sentence, choose any $\B\in\Flat_{HC}(L)$, and let $(M,cov)$ be from Proposition~\ref{reverse} such that 
the corresponding $(M,\F)^\flat\cong\B$.  
It is easily checked that the map $j:\B\rightarrow\bcan(\B)$ defined by $j(\aa)=\css_k(M,\abar)$ for some or any $\abar\in M^k$ such that $cov_k(\abar)=\aa$
is well defined and is a surjective homomorphism.  As $\bcan(\B)$ is Hausdorff from Lemma~\ref{canfacts}(1), the uniqueness of $j$ follows from the first sentence.
\hfill\qedsymbol\end{proof}

Lemma~\ref{uniquemap} gives that the function below is well defined.

\begin{definition}  Let $\cmap:\Flat_{HC}(L)\rightarrow HC$ be the function described as:  $\cmap(\B)$ is the unique surjective homomorphism $j:\B\rightarrow\bcan(\B)$.

\end{definition}

We glean one easy consequence of this.  The characterization will be useful in proving both that the class of Hausdorff structures in HC is $HC$-invariant, and that `being Hausdorff'
is absolute among forcing extensions.

\begin{lemma} \label{charHdf}   The following are equivalent for  $\B\in\Flat_{HC}(L)$.  
\begin{enumerate}
\item  $\B$ is Hausdorff;
\item   $\cmap(\B)$ is injective; and  
\item  $\cmap(\B)$ is an $\Lflat$-isomorphism.
\end{enumerate}
\end{lemma}

\begin{proof}  $(1)\Rightarrow(2)$:
Choose any $\B\in\Flat_{HC}(L)$ and denote $\bcan(\B)$ as $\B'$ and $\cmap(\B)$ as $j:\B\rightarrow\B'$.
First, assume $\B$ is Hausdorff and choose $\aa\neq\bb$ from $\B$.  As $\B$ is Hausdorff, $\flattp_{\B}(\aa)\neq\flattp_{\B}(\bb)$, so
$\flattp_{\B'}(j(\aa))\neq\flattp_{\B'}(j(\bb))$ by Lemma~\ref{keyH}.  Thus, $j(\aa)\neq j(\bb)$, so $j$ is injective.

$(2)\Rightarrow(3)$ is immediate from the definitions.

$(3)\Rightarrow(1)$:  Since 
 $\bcan(\B)$ is Hausdorff by Lemma~\ref{canfacts}(1), (3) implies 
$\B$ is as well.
\hfill\qedsymbol\end{proof}

Let $\Hdf_{HC}:=\{\B\in\Flat_{HC}(L):\B$ is Hausdorff$\}$.
We consider forcing extensions $\V[G]$ of $\V$.

\begin{lemma}  \label{allHCinvariant}
\begin{enumerate}
\item  $\can$, $\bcan$, and $\cmap$ are strongly definable and are persistently functions.
\item  Both  $\Hdf_{HC}^*$ and $\Hdf_{HC}$ are strongly definable sets.
\item  The identity map $id:\Hdf_{HC}^*\rightarrow \Flat_{HC}$ is a persistent, strongly definable cross section of $\bcan:\Flat_{HC}\rightarrow\Hdf_{HC}^*$.
\end{enumerate}
\end{lemma}

\begin{proof}  (1)  We first show $\can$ is HC-invariant and is persistently a function.
Choose any $M\in\Mod_{HC}(L)$ and let $\V[G]\supseteq\V$ be any forcing extension.  
As $\css$ is strongly definable, for every $\abar\in M^k$, we have $\css_k(M,\abar)^{\V[G]}=\css_k(M,\abar)^{\V}$.
It follows from this that $\can(M)^{\V[G]}=\can(M)^{\V}$.   Thus, $\can$ is strongly definable.
The argument that $\can$ is well defined also holds in $V[G]$, hence
$\can$ is persistently a function.

The argument for $\bcan$ is similar.  Given $\B\in\Flat_{HC}(L)$, note that any $M\in\V$ witnessing $\bcan(\B)=\can(M)$ also witnesses this in $\V[G]$.
Thus, $\bcan(\B)^{\V[G]}=\bcan(\B)^{\V}$.  That $\bcan$ is persistently a function follows by applying Lemma~\ref{samemodel} in $\V[G]$.

Now, for $\cmap$, choose $\B\in\Flat_{HC}(L)$ and let $j:=\cmap(\B)$ in $\V$.  Then easily, in $\V[G]$ we also have
$j:\B\rightarrow\bcan(\B)$ is a surjective homomorphism.  By applying Lemma~\ref{uniquemap} to $j$ in $\V[G]$ it is the only one,
hence $\cmap(\B)^{\V[G]}=j$ as well.   Thus, $\cmap$ is strongly definable and is persistently a function.

(2)  For $\Hdf_{HC}^*$ we use Lemma~\ref{canfacts}(4).  Choose $\B\in HC^{\V}$.
First, suppose  $\B\in(\Hdf_{HC}^*)^{\V}$ and let $\V[G]\supseteq\V$ be any forcing extension.
Then $\B=\can(\B)$ in both $\V$ and  $\V[G]$, so $\B\in(\Hdf_{HC}^*)^{\V[G]}$ by applying Lemma~\ref{canfacts}(4) in $\V[G]$.
Conversely, assume  $\B\in(\Hdf_{HC}^*)^{\V[G]}$.  As  $\Flat(L)$, the class of all models of $\Ax$, is absolute between forcing extensions and $\B\in HC^{\V}$,
 we know $\B\in\Flat_{HC}(L)^{\V}$.
By Lemma~\ref{canfacts}(4) in $\V[G]$, $\bcan(\B)=\B$ in $\V[G]$.  As $\bcan$ is strongly definable,
 this implies $\bcan(\B)=\B$ in $\V$ as well.  Thus, $\B\in(\Hdf_{HC}^*)^{\V}$.

Showing that $\Hdf_{HC}$ is strongly definable is similar, using  Lemma~\ref{charHdf} in both $\V$ and $\V[G]$.
On one hand, if $\B\in\Hdf_{HC}(L)$ in $\V$, then letting $j=\cmap(\B)$, $j:\B\rightarrow\bcan(\B)$ is injective, both in $\V$ and (from above) in
$\V[G]$.  Thus, $\B\in\Hdf_{HC}(L)$ in $\V[G]$ as well.
Conversely, if $\B\in HC^{\V}$ and $\B\in\Hdf_{HC}(L)^{\V[G]}$.  
Then $j:\B\rightarrow\bcan(\B)$ is injective in both $\V[G]$ and $\V$, so $\B\in \Hdf_{HC}(L)$.

(3)  This is obvious by unpacking the definitions.
\hfill\qedsymbol\end{proof}

It is noteworthy that the verification of $\Hdf^*$ being strongly definable used the characterization of it in terms of the fixed points of $\bcan$
as opposed to anything about the strongly definable function $\can$.  This asymmetry will be magnified 
in Conclusion~\ref{Hdfconc}, where we show that $\Hdf^*_\ptl$
is the image of $\bcan_\ptl$, but the image of $\can_\ptl$ might be smaller.  Indeed, this distinction manifests the difference between $\CSS(L)_{\rm sat}$ and $\CSS(L)_\ptl$
described in the Introduction.

\bigskip

We now leave the HC world and enter into the world of all sets by applying the $\ptl$ operator.
Note that as `being a structure' and `being a flat structure' are absolute,
$\Mod(L)_\ptl=\Mod(L)$, the class of all $L$-structures, and similarly, $\Flat(L)_\ptl$ is the class of all flat $\Lflat$-structures, which we denote
by $\Flat(L)$.   In light of Lemma~\ref{allHCinvariant}, and Lemma~\ref{ptl} we have the following:

\begin{itemize}
\item  $\can_{\ptl}:\Mod(L)\rightarrow \Hdf^*_\ptl$ defined as:  $\can_{\ptl}(M)=\can(M)^{\V[G]}$ for some/every forcing extension $\V[G]\supseteq\V$
with $M\in HC^{\V[G]}$;
\item  $\bcan_{\ptl}:\Flat(L)\rightarrow \Hdf^*_{\ptl}$ defined as:  $\bcan_{\ptl}(\B)=\bcan(\B)^{\V[G]}$ for some/every forcing extension $\V[G]\supseteq\V$;
with $\B\in HC^{\V[G]}$;

\item  $\cmap_{\ptl}:\Flat_{\V}(L)\rightarrow \V$ defined as:  $\cmap_{\ptl}(\B)$ is the unique $j:\B\rightarrow\bcan(\B)$, as computed in some/every forcing
extension 
 $\V[G]\supseteq\V$ with $\B\in HC^{\V[G]}$.
 \end{itemize}
 
%
 We isolate the following for emphasis.
 
 \begin{lemma} \label{onto} The class functional $\bcan_{\ptl}:\Flat(L)\rightarrow \Hdf^*_\ptl$ is a (surjective) projection.
 \end{lemma}
 
%

 \begin{proof}  This follows immediately from   Lemmas \ref{cross} and \ref{allHCinvariant}.
 \hfill\qedsymbol\end{proof}
 
 \begin{lemma}  \label{canHDF}  For any flat structure $\B$, $\bcan_\ptl(\B)$ is Hausdorff.
 \end{lemma}
 
 \begin{proof}  Fix a flat structure $\B\in\V$ and choose any forcing extension $\V[G]\supseteq\V$ such that $\B\in HC^{\V[G]}$.
By Proposition~\ref{reverse}, choose  $(M,\F)\in HC^{\V[G]}$ such that $(M,\F)^\flat\cong\B$.  Working in $\V[G]$, put $\B':=\can(M)$,
which is equal to $\bcan(\B)$.  By the strong definability of $\bcan$ and  since $\B\in\V$, $\B'\in\V$ as well.  We know that  $\B'$ is Hausdorff in $\V[G]$
by Lemma~\ref{canfacts}(1), but we must show  $\B'$ is Hausdorff in $\V$.  For this, choose any $\aa\neq\bb$ in $\B'$.
We will show that  $\flattp_{\B'}(\aa)$ and  $\flattp_{\B'}(\bb)$ differ on some formula in $\V$.    This is obvious unless both $\aa,\bb$ are in $U_k^{\B'}$ for the same $k$,
so we assume they are.  

Working in $\V[G]$, by the definition of $\can(M)$, $\aa=\css(M,\abar)$ and $\bb=\css(M,\bbar)$ for some tuples $\abar,\bbar\in M^k$, i.e., are $L(\cbar_k)_{\infty,\omega}$-sentences
$\sigma,\tau$, respectively.  As $\sigma$ and $\tau$ are distinct canonical Scott sentences,  they are explicitly contradictory.  
Thus,
the $L_{\infty,\omega}$-sentence $\forall\xbar (\hat{\sigma}(\xbar)\rightarrow\neg\hat{\tau}(\xbar))$ is persistently valid, so holds in $M$.
Now, as $\B'=\can(M)$,
$$\B'\models\hat{\sigma}^\flat(\aa)\wedge\hat{\tau}^\flat(\bb)$$
Since $M\models \forall\xbar (\hat{\sigma}(\xbar)\rightarrow\neg\hat{\tau}(\xbar))$, 
$\B'\cong \Mflat$ and Lemma~\ref{27} imply 
$$\B'\models\forall z[\hat{\sigma}^\flat(z)\rightarrow \neg \hat{\tau}^\flat(z)]$$
Thus, in $\V[G]$, 
$\B'\models\neg\hat{\tau}^\flat(\aa)\wedge\hat{\tau}^\flat(\bb)$. 
Since $\B'\in\V$ and $\aa,\bb\in\B'$, both $\sigma,\tau\in\V$, as is
 the syntactic variant $\hat{\tau}^\flat(z)\in\V$.
As satisfaction is absolute and $\B'\in \V$, we conclude that $\B'\models\neg\hat{\tau}^\flat(\aa)\wedge\hat{\tau}^\flat(\bb)$  in $\V$ as well.
Thus,  the formula $\hat{\tau}^\flat(z)$ witnesses that $\flattp_{\B'}(\aa)\neq \flattp_{\B'}(\bb)$ in $\V$.
\hfill\qedsymbol\end{proof}


\begin{lemma}  The notion of `$\B$ is Hausdorff' is absolute between forcing extensions.  Thus, $\Hdf_\ptl$ denotes the class of Hausdorff
flat structures.
\end{lemma}

\begin{proof}  Suppose $\V[G]\supseteq\V$ is any forcing extension and assume $\B\in \V$.  
That $\B$ Hausdorff in $\V$ implies $\B$ Hausdorff in $\V[G]$ is easy, since for different elements $\aa\neq \bb$
of $\B$ if $\B\models\phi^\flat(\aa)\wedge\neg\phi^\flat(\bb)$ in $\V$, then, as $\phi^\flat\in\V$,  the same holds in $\V[G]$.

For the converse, assume $\B\in\V$ is Hausdorff in $\V[G]$.  As $\B\in\V$, both $\bcan_\ptl(\B)$ and
$j:=\cmap_\ptl(\B)$ are in $\V$ by Lemma~\ref{ptl}.  Choose a forcing extension
$\V[G][H]$ in which $\B\in HC^{\V[G][H]}$.  Note that by the first paragraph, $\B$ is Hausdorff in $\V[G][H]$ as well.  
By definition of $\cmap_\ptl$, $j:\B\rightarrow\bcan_\ptl(\B)$ is equal to $\cmap(\B)$ in $\V[G][H]$.
Thus, working in $\V[G][H]$, $j$ is an $\Lflat$-isomorphism by Lemma~\ref{charHdf}.  

Now we work in $\V$.  Since $j\in\V$ and `being an isomorphism' is absolute, $\B\cong \bcan_\ptl(\B)$ in $\V$.
As $\bcan_\ptl(\B)$ is Hausdorff by Lemma~\ref{canHDF}, so is $\B$.
\hfill\qedsymbol\end{proof}

Given this absoluteness, we can now recast our definitions as follows.

\begin{definition}   Let $\Hdf$ consist of all Hausdorff flat structures $\B\in\V$ and let 
$\Hdf^*$ consist of all flat structures $\B\in\V$ such that $\bcan_\ptl(\B)=\B$.

 \end{definition}  
 
 The above agrees with our previous definitions in the sense that for any $\B\in \Flat(L)$, $\B\in\Hdf$ (resp.\ $\B\in\Hdf^*$) 
if and only if 
 $(\B\in \Hdf_{HC})^{\V[G]}$ (resp.\  $(\B\in \Hdf^*_{HC})^{\V[G]}$)
 for some/every forcing extension $\V[G]\supseteq\V$ with $\B\in HC^{\V[G]}$.
That is, $\Hdf=\Hdf_\ptl$ and $\Hdf^*=\Hdf^*_\ptl$.

\begin{lemma}  \label{flatequiv} The following are equivalent for Hausdorff flat structures $\B_1,\B_2$ (of any size):
\begin{enumerate}
\item  $\B_1\equiv_{\infty,\omega} \B_2$;
\item  $\B_1\cong\B_2$;
\item  There is a unique $\Lflat$-isomorphism $h:\B_1\rightarrow\B_2$;
\item  $\bcan_\ptl(\B_1)=\bcan_\ptl(\B_2)$.

\end{enumerate}
\end{lemma}

\begin{proof}   As $(3)\Rightarrow(2)\Rightarrow(1)$ is obvious, it suffices to prove $(1)\Rightarrow(4)$ and $(4)\Rightarrow(3)$.
For both parts, choose a forcing extension $\V[G]\supseteq\V$ in which both $\B_1,\B_2$ are in $HC^{\V[G]}$.

$(1)\Rightarrow(4)$:  Assume $\B_1\equiv_{\infty,\omega} \B_2$.  Then, classically, $\B_1\cong\B_2$ in $\V[G]$.  
Thus, by Lemma~\ref{canfacts}(2), $\bcan(\B_1)=\bcan(\B_2)$ in $\V[G]$.  
But,  by definition of $\bcan_\ptl$, we have $\bcan_\ptl(\B_\ell)=\bcan(\B_\ell)^{\V[G]}$ for $\ell=1,2$, hence
$\bcan_\ptl(\B_1)=\bcan_\ptl(\B_2)$.

$(4)\Rightarrow(3)$:  By definition of $\bcan_\ptl$, we have $\bcan_\ptl(\B_\ell)=\bcan(\B_\ell)^{\V[G]}$ for $\ell=1,2$, hence
$\bcan(\B_1)=\bcan(\B_2)$ in $\V[G]$.  Thus, by Lemma~\ref{canfacts}(2), there is some $\Lflat$-isomorphism $h:\B_1\rightarrow\B_2$ in $\V[G]$.  
We need to show that $h\in\V$, and moreover, there cannot be any other $\Lflat$-isomorphism.  For both of these, note that since both $\B_1$ and $\B_2$
are Hausdorff, distinct elements have distinct $\flattp$'s.  Using $h$, we see that  in $\V$, the sets $\{\flattp(\aa):\aa\in\B_1\}$ and $\{\flattp(\bb):\bb\in\B_2\}$
are equal.  Let $h_0\in \V$ be the unique bijection $h_0:\B_1\rightarrow\B_2$ that is $\flattp$-preserving, i.e., $\flattp(h_0(\aa))=\flattp(\aa)$ for all $\aa\in\B_1$.
As the uniqueness of $h_0$ is persistent in forcing extensions, it follows that $h_0=h$.  Thus, $h\in\V$ and is the unique isomorphism between $\B_1$ and $\B_2$.
\hfill\qedsymbol\end{proof}

From this, we get the following useful corollary.

\begin{Corollary}  \label{relateflats}
 Let $M$ be any $L$-structure and let $\Mflat$ be the canonical flattening of $M$.  Then for any $L$-structure $N$,
$$M\equiv_{\infty,\omega} N\quad\hbox{if and only if} \quad N^\flat\cong M^\flat$$
\end{Corollary}

\begin{proof}  Say $M\equiv_{\infty,\omega} N$.  Pass to any forcing extension $\V[G]$ in which both $M$ and $N$ are countable, hence $M\cong N$ in $\V[G]$.
As back-and-forth equivalence is absolute,  $\Mflat\cong N^\flat$ in $\V[G]$.  Pulling back, we have $\Mflat\equiv_{\infty,\omega} N^\flat$ in $\V$.  But, as both $M^\flat$ and $N^\flat$ are Hausdorff,
there is a unique isomorphism between them by Lemma~\ref{flatequiv}.

Conversely, suppose $N^\flat\cong M^\flat$, say via the map $h$, which is unique by Lemma~\ref{flatequiv}. Let $\F=\{(\abar,\bbar)\in N^{<\omega}\times M^{<\omega}: h(\pi(\abar))=\pi(\bbar)\}$,
and it suffices to show that $\F$ is a back-and-forth system between $N$ and $M$.  But this follows easily from Lemma~\ref{stepup}.
\hfill\qedsymbol\end{proof}

\begin{lemma}  \label{oneside}  There is a class bijection $k:\CSS(L)_\ptl\rightarrow \Hdf^*$ given by
$k(\phi)=\can(M)$, where $M\models\phi$, $M\in HC^{\V[G]}$ for some forcing extension $\V[G]$ of $\V$ for which $\phi\in HC^{\V[G]}$.
\end{lemma}

\begin{proof}  We first show that the functional $k$ is well-defined and that $k(\phi)\in\V$.  For both, given $\phi\in\CSS(L)_\ptl$,
say $\V[G]$ and $\V[H]$ are each forcing extensions of $\V$ for which $\phi\in HC^{\V[G]}\cap HC^{\V[H]}$.  Choose countable models
$M\models\phi$ in $\V[G]$ and $N\models\phi$ in $\V[H]$.  Then both $M,N$ are elements of the product forcing $\V[G\times H]$.  
Working in $\V[G\times H]$, since $\phi$ is a canonical Scott sentence,
$M\cong N$, hence $\can(M)=\can(N)$ by Lemma~\ref{samemodel}.  Let $\B$ denote this common value.  Clearly, $\can^\flat(\B)=\B$, so $\B\in \Hdf^*$ (in $\V[G\times H]$).  
Now, by e.g., Lemma~2.5 of \cite{URL}, it follows that $\B\in\V$, hence $\B\in \Hdf^*$ (in $\V$) by the absoluteness of $\Hdf^*$.  So put $k(\phi):=\B$.

To see that $k$ is bijective,  we describe its inverse.  
Choose any $\B\in \Hdf^*$.  As $\bcan_{\ptl}(\B)=\B$, $U_0^{\B}=\{\phi\}$, where $\phi=\css(M)$ for some/every $M\in HC_{\V[G]}$ for some/every forcing
extension $\V[G]$ in which $\B\in HC^{\V[G]}$.  Then $k(\phi)=\B$.
\end{proof}

%

The following summarizes the key relationships.  The map between (1) and (2) is given by Lemma~\ref{flatequiv} and the map between (3) and (2) is given by Lemma~\ref{oneside}.

\begin{Conclusion}  \label{Hdfconc}  There are natural bijections between three classes:
\begin{enumerate}
\item  $(\Hdf,\cong)$, the class of isomorphism types of Hausdorff flat structures;
\item  The class $\Hdf^*$; and
\item  The class $\CSS(L)_\ptl$ discussed in \cite{URL}.
\end{enumerate}
\end{Conclusion}  

By contrast, the subclass of $\Hdf^*$ consisting of all  $\can_\ptl(M)$ for $M\in\Mod(L)$ is less interesting.  In the notation of \cite{URL}, it is bijective with
$\CSS(L)_{\rm sat}$.

\subsection{Relativizing to $\Mod(\Phi)$}  In the whole of the previous subsection, we were working with the class of all $L$-structures and all flat $\Lflat$-structures.
However, for applications we want to restrict to e.g., models of a particular sentence of $L_{\omega_1,\omega}$.
For a fixed $\Phi\in L_{\omega_1,\omega}$, let $\Mod(\Phi)$  and $\Flat(\Phi)$ denote the class of models of $\Phi$ and flat models of $\Phi^\flat$, respectively.
Similarly, $\Hdf(\Phi)$ denotes the class of Hausdorff flat structures that model $\Phi^\flat$ and $\Hdf^*(\Phi)$ is the image of $\Flat(\Phi)$ under
$\bcan_\ptl$.  
Owing to the fact that satisfaction is absolute, all of the results in the previous subsection go through, leading to the following  relativization of Conclusion~\ref{Hdfconc}.

\begin{theorem}  \label{restrict}  Let $\Phi\in L_{\omega_1,\omega}$.  There are natural bijections between three classes:
\begin{enumerate}
\item  $(\Hdf(\Phi),\cong)$, the class of isomorphism types of Hausdorff flat models of $\Phi$;
\item  The class $\Hdf^*(\Phi)$; and
\item  The class $\CSS(\Phi)_\ptl$ discussed in \cite{URL}.
\end{enumerate}
\end{theorem}  

In \cite{URL},  a sentence $\Phi\in L_{\omega_1,\omega}$ is {\em grounded} if $\CSS(\Phi)_{\rm sat}=\CSS(\Phi)_\ptl$.
Here, we obtain a more malleable equivalence.  Recall that by Proposition~\ref{reverse}, every {\em countable} Hausdorff flat structure
is isomorphic to $M^\flat$ for some $M\in\Mod(\Phi)$.

\begin{Corollary}  A sentence $\Phi\in L_{\omega_1,\omega}$ is grounded if and only if every Hausdorff flat structure (of any size) is isomorphic to
$M^\flat$ for some $M\in\Mod(\Phi)$.
\end{Corollary}

\section{Flat structures, sharp expansions of a countable $M$, and $\Aut(M)$}  \label{precise}
Let $M$ be a countable $L$-structure with universe $\omega$ and let $\phi_M\in L_{\omega_1\omega}$ be any Scott sentence for $M$.
Among all flattenings $(M,\F)^\flat$, we know  there is a unique (up to unique isomorphism)
countable, Hausdorff flat structure $\B$, namely the canonical flattening $\Mflat$ of $M$.  As $M\models \phi_M$, it follows that $\B\models \phi_M^\flat\wedge\Ax$.
We will see that the other countable models of $\phi_M^\flat\wedge\Ax$  correspond to closed subgroups of $\Aut(M)$ and we analyze the complexity of this correspondence.

In what follows, for $\Phi\in L_{\omega_1,\omega}$, $\Mod_\omega(\Phi)$ denotes the Polish space of models of $\Phi$ with universe $\omega$.
Note that $(\Mod_\omega(\phi_M\wedge\Axsharp),\cong_\Lsharp)$ is HC-bireducible to $\{$sharp expansions $(M,\F)\}/\cong_{\Lsharp}$.

From Proposition~\ref{reverse}  there
is an HC-bireduction between $\Mod_\omega(\phi_M^\flat\wedge\Ax)/\cong_{\Lflat}$ and $\{$sharp expansions $(M,\F)\}/\cong_{\Lsharp}$, so to obtain a correspondence between
$\Mod_\omega(\phi_M^\flat\wedge\Ax)/\cong_{\Lflat}$ and closed subgroups of $\Aut(M)$,
it 
would suffice to find an HC-invariant bijection between sharp expansions $(M,\F)$ of $M$ and closed subgroups 
of $\Aut(M)$.  However, such a mapping cannot exist, for the simple reason that closed subgroups of $\Aut(M)$ are rarely elements of $HC$.  Thus, we instead derive a correspondence between sharp expansions and {\em codes} of closed subgroups of $\Aut(M)$.  

\subsection{Codes of closed subgroups}
We begin by defining the codes.  It is natural to work with the topological group $S_\infty=\Sym(\omega)$, which we can view as $\Aut(\E)$, where $\E:=(\omega,=)$.

\begin{definition}  For each $n\in\omega$, let $I_n=\omega^n\times \omega^n$ and let $I=\bigcup\{I_n:n\in\omega\}$.  Applying  any $\sigma\in S_\infty$ coordinatewise
induces a natural map$:I\rightarrow I$ (also called $\sigma$).
For each $(\abar,\bbar)\in I$, let
$U_{\abar,\bbar}=\{\sigma\in S_\infty:\sigma(\abar)=\bbar\}$.

\end{definition}

Then $\{U_{\abar,\bbar}:(\abar,\bbar)\in I\}$ is the standard basis for the topology of pointwise convergence on $S_\infty$.  
With respect to this topology, for any closed $C\subseteq S_\infty$, let
$$\F(C)=\{(\abar,\bbar)\in I: C\cap U_{\abar,\bbar}\neq\emptyset\}$$
denote the {\em code of $C$}.  Note that the mapping $C\mapsto\F(C)$ is injective on closed sets and $\F(C)\in HC$ for every $C\subseteq S_\infty$.
Moreover, recalling Definitions~\ref{bf}~and~\ref{sharp}, $\F(C)$ is a downward closed back and forth system on $\E$ for every closed $C\subseteq S_\infty$.

We endow $\P(I)$ with the usual Cantor topology, i.e.,  sets  of the form  $V_{\abar,\bbar}:=\{A\subseteq I:(\abar,\bbar)\in A\}$ and their complements form a sub-basis for the topology.
As $I$ is countable, this topology is Polish.

Now suppose $\F\subseteq I$ is a downward closed back and forth system.  Construing $\sigma\in S_\infty$ as $\sigma:I\rightarrow I$, the inverse operation is given by
$$C_{\F}=\{\sigma\in S_\infty: \sigma\subseteq\F\}$$

The proof of the following lemma is routine, once one notes that for  $\F$ any downward closed back and forth system, if $(\abar,\bbar)\in\F$ then there is some
$\sigma\in S_\infty$ with $\sigma\subseteq \F$ and $\sigma(\abar)=\bbar$.

\begin{lemma}  \label{verify}
The set $\{\F(C):\emptyset\neq C\subseteq S_\infty$ closed$\}$ is equal to the set of downward closed back and forth systems on $\E$.  Specifically, 
$C_{\F(C)}=C$ for every non-empty closed $C\subseteq S_\infty$, and
$\F(C_{\F})=\F$ for every downward closed back and forth system $\F$ on $\E$.
\end{lemma}


It is easily seen from Lemma~\ref{verify} that the subspace $\{\F(C):C$ closed in $S_\infty\}\subseteq\P(I)$ is closed, hence is also a Polish space.
The correspondence given in Lemma~\ref{verify} explains our interest in Definitions~\ref{bf} and \ref{sharp}. 
Continuing, we are not so much interested in arbitrary closed sets $C$, but rather 
closed subgroups $H\le \Aut(\E)$.  In the proof of Lemma~\ref{subgroup}, 
the three clauses of being a subgroup exactly align with the reflexivity, symmetry, and transitivity clauses of being an equivalence relation.

\begin{lemma}  \label{subgroup}  A closed subset $H\subseteq \Aut(\E)$ is a closed subgroup if and only if $\F(H)$ is a sharp back and forth system on the structure $\E$.
\end{lemma}

Relativizing Lemma~\ref{subgroup} to closed subgroups of $\Aut(M)$ for a countable structure $M$  yields  the following fundamental result.

\begin{proposition} \label{link} Let $M\in HC$ be any structure.  The map $(M,\F)\mapsto \F$ is a  persistent, strongly definable bijection between
$\{$sharp expansions of $M\}$ and $\{$codes of closed subgroups of $\Aut(M)\}$.
\end{proposition}

\subsection{Conjugacy and sharp expansions}

Fix an $L$-structure $M$ with universe $\omega$.  Classically, any expansion $M^*$ of $M$ can be `seen' by looking at $\Aut(M^*)$, which is a closed subgroup of $\Aut(M)$.
Conversely, any closed subgroup $H\le \Aut(M)$ gives rise to an expansion $M_H$ by letting $L^*=L\cup\{P^{\abar}(\xbar):\abar\in M^{<\omega}\}$ and letting $M^*$ be the expansion of $M$ formed by interpreting each $P^{\abar}$ as the $H$-orbit of $\abar$ in $M^{\lg(\abar)}$.  In such an expansion we have $\Aut(M_H)=H$.

Proposition~\ref{link} gives a different correspondence between sharp expansions and closed subgroups.
Given a sharp expansion $(M,\F)$ of $M$, the relevant closed subgroup is
$$\Fix(M,\F)=\{\sigma\in \Aut(M):(M,\F)\models E_k(\abar,\sigma(\abar))\ \hbox{for all $k$, $\abar\in M^k$}\}$$
which is a closed subgroup of $\Aut(M,\F)$, but is typically not equal to it.  The relevance of $\Fix(M,\F)$ is largely explained by the following easy lemmas.

\begin{lemma}  \label{tech}  Let $M$ be any  $L$-structure with universe $\omega$ and let $\F$ be any sharp back-and-forth system on $M$.  
Then, for any $(\abar,\bbar)\in \F$ (i.e., $(M,\F)\models E_k(\abar,\bbar)$) there is some $\sigma\in \Fix(M,\F)$ such that
$\sigma(\abar)=\bbar$.
\end{lemma}  

\begin{proof} An appropriate $\sigma$ can be constructed by finite approximations using the back-and-forth property of $\F$.
\hfill\qedsymbol\end{proof}

\begin{lemma}  \label{Fix}  For any sharp back and forth system $\F$ on $M$, $\Fix(M,\F)=C_{\F}$ from Lemma~\ref{verify}.
\end{lemma}

\begin{proof}  Choose $\sigma\in \Fix(M,\F)$.  To see $\sigma\subseteq\F$, choose any $\abar\in M^k$.  Since $(M,\F)\models E_k(\abar,\sigma(\abar))$, we have
$(\abar,\sigma(\abar))\in\F$.  Conversely, choose $\sigma\in C_{\F}$.  To see that $\sigma\in \Aut(M)$, choose any $(\abar,\bbar)\in \sigma$.
Since $(\abar,\bbar)\in\F$, it follows from Lemma~\ref{tech} that there is an automorphism $\psi\in \Aut(M)$ with $(\abar,\bbar)\in \psi$.  As this holds for all $(\abar,\bbar)$ and since $\Aut(M)$ is
closed in $S_\infty$,
we conclude $\sigma\in \Aut(M)$.  To see $\sigma\in \Fix(M,\F)$,  choose any $\abar\in M^k$.  Since $(\abar,\sigma(\abar))\in \F$, we have $(M,\F)\models E_k(\abar,\sigma(\abar))$,
as required.
\hfill\qedsymbol\end{proof}

\begin{definition} \label{CM}  Given an $L$-structure $M$ with universe $\omega$, let 
$$\C_M=\{\F(H)\in\P(I):\hbox{$H$ a closed subgroup of $\Aut(M)\}$.}$$
Conjugacy induces an equivalence relation $\sim_{\conj}$ on $\C_M$ as:  $\F(H_2){\sim}_{\conj} \F(H_1)$ if and only if 
$H_2$ and $H_1$ are $\Aut(M)$-conjugate, i.e., $H_2=\delta H_1\delta^{-1}$ for some $\delta\in \Aut(M)$.

\end{definition}

In light of Lemma~\ref{verify}, $\C_M$ is the set of sharp back and forth systems on $M$.  
In terms of descriptive set theory, being a closed subspace of $\P(I)$, $\C_M$ is a Polish space.  Thus,  being an orbit equivalence relation of the action of a Polish group,
$\sim_{\conj}$ is an analytic equivalence relation on $\C_M$.
We summarize our correspondence with the following.  Recall the definitions of $\Axsharp$ and $\Ax$ from Definitions~\ref{sharpdef} and \ref{flatdef}, respectively.

\begin{proposition} \label{bigcompare} Let $M$ be any $L$-structure with universe $\omega$ and let $\phi_M\in L_{\omega_1,\omega}$ be any Scott sentence for $M$.
The following three strongly definable quotients are HC-bireducible.
\begin{enumerate}  
\item $(\Mod_\omega(\phi_M^\flat\wedge\Ax),\cong_{\Lflat})$;
\item  $(\Mod_\omega(\phi_M\wedge\Axsharp),\cong_{L^\sharp})$; and 
\item $(\C_M,\sim_\conj)$.
\end{enumerate}
\end{proposition}

\begin{proof}  The correspondence between (1) and (2) is given by the flattening operation and Proposition~\ref{reverse}.
To get the HC-bireduction between (2) and (3), we use the map $(M,\F)\mapsto \F$ from Proposition~\ref{link}.  We show this map preserves quotients in both directions, i.e.
$$(M,\F_1)\cong_{L^\sharp} (M,\F_2)\quad \hbox{if and only if} \quad \F_1\sim_\conj \F_2$$
To establish this,  choose sharp back and forth systems $\F_1,\F_2$ on $M$.  By Lemma~\ref{Fix}, $H_1:=\Fix(M,\F_1)$ and $H_2:=\Fix(M,\F_2)$ are the closed subgroups of $\Aut(M)$ coded by $\F_1$ and $\F_2$, respectively.   It is readily checked that for any $\delta\in\Aut(M)$, $H_2=\delta^{-1} H_1\delta$ if and only if $\F_2=\delta^{-1} \F_1\delta$ if and only if $\delta:(M,\F_1)\rightarrow (M,\F_2)$ is an $L^\sharp$-isomorphism. \hfill\qedsymbol\end{proof}

%

\subsection{Comparing families $\C_M$ and the completeness of $\C_{(\omega,=)}$}

Here we  compare the complexity of quotients of $\C_M$ and $\C_N$ by ${\sim}_{\conj}$, where $M$ and $N$ are each structures with universe $\omega$, although possibly in different languages.
 The following Lemma must be very standard.

\begin{lemma} \label{algebra}   Suppose $M$ and $N$ are structures (possibly in different countable languages), each with universe $\omega$, and choose
 a continuous, surjective homomorphism $f:\Aut(N)\rightarrow \Aut(M)$.   Then the map $H\mapsto f^{-1}(H)$ is an injection of closed subgroups that respects conjugacy.
Thus, the induced mapping  $h:(\C_M,\sim_\conj)\rightarrow (\C_N,\sim_\conj)$ of the codes given by $h(\F(H))=\F(f^{-1}(H))$ is an HC-reduction.
\end{lemma}

\begin{proof}  We first show that the map $H\mapsto f^{-1}(H)$  factors through conjugacy.  That is, suppose $H_1,H_2\in \C_M$ satisfy $H_2=\delta H_1 \delta^{-1}$ for some $\delta\in \Aut(M)$.  
As $f$ is onto, choose $\gamma\in \Aut(N)$ such that $f(\gamma)=\delta$.  We argue that $f^{-1}(H_2)=\gamma f^{-1}(H_1) \gamma^{-1}$.  
To see this, choose $\sigma\in f^{-1}(H_1)$.  It suffices to show $\gamma\sigma\gamma^{-1}\in f^{-1}(H_2)$, i.e., that $f(\gamma\sigma\gamma^{-1})\in H_2$.
But this is immediate, as $f$ is a homomorphism and $\delta f(\sigma) \delta^{-1}\in H_2$.  Thus, $h$ is well-defined.

To see that $h$ is injective on classes, suppose $G_1=f^{-1}(H_1)$ and $G_2=f^{-1}(H_2)$ in  $\C_N$ are conjugate, i.e., $G_2=\gamma G_1 \gamma^{-1}$ for some $\gamma\in \Aut(N)$.
As $f$ is a homomorphism, it follows that $H_2=f(\gamma) H_1 f(\gamma)^{-1}$, so $H_1$ and $H_2$ are conjugate in $\C_M$.

To show HC-invariance, choose a  closed $H\le \Aut(M)$  and   a forcing extension $\V[G]\supseteq\V$.  
Let $H'$ be the closed subgroup of $\Aut(M)^{\V[G]}$ with the same code as $H$.  Since they have the same codes, $H\subseteq H'$ is dense.
Because analytic statements are absolute,  we have  $f^{-1}(H)^{\V}$ is dense in $f^{-1}(H')^{\V[G]}$, so they have the same codes.
\hfill\qedsymbol\end{proof}

For this to be interesting, we need to know that $(S_\infty,\conj)$ is sufficiently complicated.  Recall that we can identify $S_\infty$ with $\Aut(\omega,=)$.

\begin{lemma}  \label{FSidea}  $(\Mod_\omega(\mbox{Graphs}),\cong)$ is  HC-reducible to $(\C_{(\omega,=)},\sim_\conj)$.
\end{lemma}

\begin{proof}  A graph is an $\{R\}$-structure, where $R$ is a symmetric, irreflexive relation.
Call an $\omega$-tree $(T,\trianglelefteq)$ {\em padded} if, for all $a\in T$ and $b\in \Succ(a)$, there is an infinite set $\{b_i:i\in\omega\}\subseteq \Succ(a)$
such that $T_{\ge b_i}\cong T_{\ge b}$ for every $i$.   An easy alteration\footnote{Specifically, adding multiplicities to account for the padding.}
 of the original Friedman-Stanley proof that $\{$Subtrees of $\omega^{<\omega}\}$ are Borel complete
(Theorem~1.1.1 of \cite{FS})
shows that there is a Borel reduction 
$$f:\ \mbox{Graphs}\rightarrow \ \hbox{Padded subtrees of $(\omega\times\omega)^{<\omega}$}$$
Recall $\E=(\omega,=)$, so $\Aut(\E)=S_\infty$ and $\C_{\E}$ denotes the codes of closed subgroups of $S_\infty$.
Let $g:\{\mbox{Graphs}\} \rightarrow \C_{\E}$ be the map $g(G):=\F(\Aut(f(G)))$.

We must show that if $G,H$ are graphs with universe $\omega$ and the groups $\Aut(f(G)), \Aut(f(H))$ are conjugate by some $\delta\in S_\infty$,
then $G$ and $H$ are isomorphic. [The converse is easy.] We verify this by way of the following Claim.

\medskip
\noindent{\bf Claim.}  If $(T,\trianglelefteq)$ is any padded $\omega$-tree, then for any $a,b\in T$, $a\trianglelefteq b$ if and only if every $\sigma\in \Aut(T,\trianglelefteq)$ that fixes $b$ also fixes $a$.
\medskip

\begin{proof}  Left to right is obvious, as for any $\omega$-tree $(T,\trianglelefteq)$, if $a\trianglelefteq b$, then $a\in\dcl(b)$.
For the converse, note that for any $c\in T$ and any $d,d'\in \Succ(c)$ with $T_{\ge d}\cong T_{\ge d'}$, there is $\sigma\in \Aut(T,\trianglelefteq)$ fixing 
$T\setminus (T_{\ge d}\cup T_{\ge d'})$
 with $\sigma(d)=d'$.
 
Now
assume that $a\not\trianglelefteq b$ and let $c=a\wedge b$.  There are  two cases.  First, if $c=b$ (i.e., if $b\trianglelefteq a$) then let $d\in \Succ(b)$
with $d\trianglelefteq a$.  As $(T,\trianglelefteq)$ is padded, choose $d'\in \Succ(b)$ with $d'\neq d$ but $T_{\ge d'}\cong T_{\ge d}$.  Then any 
$\sigma\in \Aut(T,\trianglelefteq)$ fixing $b$ with $\sigma(d)=d'$ will move $a$ as well.
Second, assume $c\neq b$, hence $c\triangleleft b$.  Let $d\in \Succ(c)$ with $d\trianglelefteq a$.  Choose $d'\in \Succ(c)$ such that $d'\neq d$ and 
$d'\not\trianglelefteq b$, but with $T_{\ge d'}\cong T_{\ge d}$.  Any automorphism $\sigma\in \Aut(T,\trianglelefteq)$  with $\sigma(d)=d'$
that fixes $T\setminus (T_{\ge d}\cup T_{\ge d'})$ will fix $b$, but move $a$.
\hfill\qedsymbol\end{proof}

To see that this Claim suffices, suppose $G,H$ are graphs with universe $\omega$ and $\Aut(f(H))=\delta \Aut(f(G)) \delta^{-1}$.  Because of the Borel reduction, it suffices to show that 
 $\delta:(f(G),\trianglelefteq)\cong (f(H),\trianglelefteq)$ is a tree isomorphism.  As $\delta\in S_\infty$ and by symmetry, it suffices to show that if $a,b\in f(G)$
 and $(f(G),\trianglelefteq)\models a\trianglelefteq b$, then $(f(H),\trianglelefteq)\models \delta(a)\trianglelefteq \delta(b)$.  But this is clear: Let $c=\delta(a)$, $d=\delta(b)$,
 and choose any $\tau\in \Aut(f(H),\trianglelefteq)$ with $\tau(d)=d$.  
Put $\sigma:=\delta^{-1}\tau\delta\in \Aut(f(G),\trianglelefteq)$.  Since $\tau(d)=d$, $\delta(\sigma(b))=\delta(b)$, hence $\sigma(b)=b$.
By the Claim, $\sigma(a)=a$ as well, hence $\tau(c)=c$.   Thus, $(f(H),\trianglelefteq)\models c\trianglelefteq d$ by the Claim.
\hfill\qedsymbol\end{proof}

\section{Characterizing Borel complete expansions}  \label{BCexpansions}

With the extensive preliminaries out of the way, we are now able to obtain our characterization of sentences of $L_{\omega_1,\omega}$ that have Borel complete expansions.
We begin with a classical result of Morley \cite{Morley} that characterizes which sentences of $L_{\omega_1,\omega}$ have Ehrenfeucht-Mostowski models.
The novelty will be that we apply this result to a class of flat structures.    The following fact is contained in Chapter~5 of \cite{Marker}.

\begin{Fact}  \label{Morley}  The following are equivalent for  $\psi\in L_{\omega_1,\omega}$.
\begin{enumerate}
\item  For all $\alpha<\omega_1$, $\psi$ has a model of size at least $\beth_\alpha$;
\item  $\psi$ has arbitrarily large models;
\item There is a countable $L^*\supseteq L$, a countable, Skolemized fragment $\F^*$ containing $\psi$, and a countable $L^*$-structure $N^*\models\psi$
that is the Skolem hull of a sequence $\<a_q:q\in\Q\>$ of $\F^*$-order indiscernibles; and
\item  For all linear orderings $(J,\le)$, there is a model $N_J\models \psi$ containing  a sequence $\<a_j:j\in J\>$ of order indiscernibles.
\end{enumerate}
\end{Fact}

For brevity, we say that {\em $\psi$ admits Ehrenfeucht-Mostowski models} if any of the four conditions hold.
Next, we have a group theoretic notion.

\begin{definition}   We say that the topological group {\em $H$ divides the topological group $G$} if there is a closed subgroup $G'\le G$
and a continuous, surjective homomorphism \hbox{$h:G'\rightarrow H$.}

\end{definition}

Obviously, if $H\le G$ is closed, then $H$ divides $G$.
We note two useful examples of dividing among automorphism groups.

\begin{lemma}  \label{connectgroups}  
\begin{enumerate}
\item  For $N^*$ chosen as in Fact~\ref{Morley}(3), $\Aut(\Q,\le)$ divides $\Aut(N^*)$.
\item Let $(M,\F)$ be any countable sharp structure and let $\B=(M,\F)^\flat$ be its flattening.  Then $\Aut(\B)$
divides $\Aut(M)$.
\end{enumerate}
\end{lemma}

\begin{proof}  (1)  Given $N^*$ as in Fact~\ref{Morley}(3), let $G'\le \Aut(N^*)$ denote the setwise stabilizer of $\{a_q:q\in\Q\}$.
Then, as $\<a_q:q\in\Q\>$ are order indiscernible and $N^*=\dcl(\{a_q:q\in\Q\})$, $G'$ is topologically isomorphic to $\Aut(\Q,\le)$.

(2)  As $\Aut(M,\F)$ is a closed subgroup of $\Aut(M)$, it suffices to verify that the homomorphism $h:\Aut(M,\F)\rightarrow \Aut(\B)$ is a continuous surjection, where
$h$ is
defined by $h(\sigma)$ being the automorphism of $\B$ satisfying $h(\sigma)(\aa)=\pi_k(\sigma(\abar))$ for every $\aa\in U_k(\B)$ and some  $\abar\in \pi_k^{-1}(\aa)$.
Showing that $h$ is well defined amounts to verifying that $\pi_k(\sigma(\abar))=\pi_k(\sigma(\abar'))$ whenever $(\abar,\abar')\in\F$, which is immediate since 
$\sigma\in \Aut(M,\F)$.   That $h$ is continuous follows from the fact that for any $\aa\in\B$, $h^{-1}(\Fix(\aa))$ is equal to the set of $\sigma\in \Aut(M,\F)$ fixing the
$E_k$-class $[\abar]_k$.
That $h$ is onto follows from the final sentence of  Proposition~\ref{reverse}:  Any $\sigma\in\Aut(\B)$ lifts to an automorphism of $(M,\F)$.
\hfill\qedsymbol\end{proof}

It is easily verified that the relation of dividing is transitive.  What is more surprising is the following, which largely appears as Lemma~2.5 of \cite{Hjorth}.

\begin{lemma}  \label{QSsame} $S_\infty$ divides $\Aut(\Q,\le)$.   For  any topological group $G$, $S_\infty$ divides $G$ if and only if $\Aut(\Q,\le)$ divides $G$.
\end{lemma}  

\begin{proof}  To see that $S_\infty$ divides $\Aut(\Q,\le)$, choose a partition $\Q=\sqcup \{D_n:n\in\omega\}$ into dense sets and let $E\subseteq \Q^2$ be the equivalence
relation given by $E(a,b)$ iff $\bigwedge_{n\in\omega} D_n(a)\leftrightarrow D_n(b)$.  Let $G'=\Aut(\Q,\le,E)$, which is  a closed subgroup of $\Aut(\Q,\le)$.
But it is easily checked that the map $h:G'\rightarrow S_\infty$, where $h(\sigma)$ is the element of $S_\infty$ given by $h(\sigma)(n)=$ the (unique) $m$ such
that $\sigma(a)\in D_m$ for any/all $a\in D_n$ is a continuous group homorphism.  The fact that $h$ is surjective is by a standard back-and-forth argument.

As $\Aut(\Q,\le)$ is a closed subgroup of $S_\infty$, it trivially divides it, so the second sentence follows by the transitivity of dividing.
\hfill\qedsymbol\end{proof}

We can now state and prove our main theorem.
Recall that by Definition~\ref{flatdef} flat $\Lflat$-structures are axiomatized by $\Ax$, hence $\Mod(\phi^\flat\wedge\Ax)=\Flat(\phi)$ for any infinitary $L$-sentence $\phi$.

\begin{theorem}  \label{HjorthThm}  For a countable language $L$, the following are equivalent for a sentence $\phi\in L_{\omega_1,\omega}$.
\begin{enumerate}
\item  $\phi$ has a Borel complete expansion (i.e., there is some countable $L'\supseteq L$ such that $\Mod_{L'}(\phi)$ is Borel complete);
\item  $\phi^\flat\wedge \Ax$ has arbitrarily large models;
\item  $\phi^\flat\wedge\Ax$ admits Ehrenfeucht-Mostowski models;
\item  $S_\infty $ divides $\Aut(M)$ for some countable $M\models \phi$.
\end{enumerate}
\end{theorem}

\begin{proof}  The equivalence $(2)\Leftrightarrow(3)$ is by Fact~\ref{Morley} applied
to  the $L_{\omega_1,\omega}$-sentence $\phi^\flat\wedge \Ax$.  We prove the equivalence of $(1)$ and $(4)$ by independently showing that each is equivalent to (2).

$(1)\Rightarrow(2)$:  
Suppose 
that the class of $L'$-models of $\phi$ is Borel complete.  To ease notation, let $\psi$ denote the $(L')_{\omega_1,\omega}$-sentence $\phi$.
As the class of $L'$-structures modeling $\psi$
 is Borel complete, 
$||\psi||=|\psi_\ptl|=\infty$, i.e.,
there is a proper class of pairwise  back-and-forth inequivalent models of $\psi$.  
Thus, $\psi^\flat\wedge\Axprime$ has a proper class of non-isomorphic models by Corollary~\ref{relateflats}.
This implies that $\psi^\flat\wedge\Axprime$ has arbitrarily large models.  
Taking $L$-reducts, we conclude that $\phi^\flat\wedge\Ax$ has arbitrarily large models as well.

$(2)\Rightarrow(1)$:  Suppose $\phi^\flat\wedge\Ax$ has arbitrarily large models, hence admits EM models.  
Applying Fact~\ref{Morley}(3) to $\phi^\flat\wedge\Ax$, choose a 
 a countable model $\B^*\models \phi^\flat\wedge\Ax$ in an expanded language $L^*\supseteq \Lflat$, 
 the Skolem hull of   a sequence $\<\aa_q:q\in\Q\>$ of order indiscernibles.
Let $\B$ be the $\Lflat$-reduct of $\B^*$, which is a countable, flat structure.
Using Proposition~\ref{reverse}, let  $(M,\F)$ be a sharp structure whose flattening is $\B$.  So $M$ is a countable $L$-structure and $M\models\phi$ since $\B\models\phi^\flat$.  

For definiteness, suppose $\<\aa_q:q\in\Q\>$ are from $U_k(\B)$.
As in the proof of Lemma~\ref{QSsame},  partition $\Q=\sqcup\{D_n:n\in\omega\}$ into dense pieces and, for each $n$, let $X_n:=\bigcup\{cov^{-1}_k(\aa_q):q\in D_n\}$.
Let $X^*:=\bigcup\{X_n:n\in\omega\}\subseteq M^k$ and let $E^*$ be the equivalence relation on $X^*$ whose classes are $X_n$.
Let $L'=L\cup\{X^*,E^*,R\}$ and let $\psi$ be the same sentence $\phi$, but viewed as an $(L')_{\omega_1,\omega}$-sentence.
We can see that $\psi$ is Borel complete directly:  Given any directed graph $G=(\omega,r)$, let $M_G$ be the expansion  $(M,X^*,E^*,R)$ of $M$ with
$R(\abar,\bbar)$ holding if and only if for some $n,m\in\omega$, $\abar\in X_n$, $\bbar\in X_m$, and $G\models r(n,m)$.
It is evident that if $M_{G_1}\cong M_{G_2}$, then $G_1\cong G_2$ as graphs.  For the other direction, choose a graph isomorphism $h:G_1\rightarrow G_2$.
As each $D_n$ is a countable, dense subset of $(\Q,\le)$,
  a back-and-forth argument yields an automorphism $h^*\in\Aut(\Q,\le)$ encoding $h$, i.e., 
  $$\hbox{for all $n\in\omega$ and for all $q\in\Q$, $[q\in D_n \Longleftrightarrow h^*(q)\in D_{h(n)}]$}$$
As $\B$ is the reduct of a Skolem hull indexed by $(\Q,\le)$, there is an $\Lflat$-automorphism $\sigma\in\Aut(\B)$
satisfying $\sigma(\aa_q)=\aa_{h^*(q)}$ for every $q\in\Q$.  By Proposition~\ref{reverse} there is an $L$-automorphism $j:M\rightarrow M$ such that for all $\ell\in\omega$
and $\bbar\in M^\ell$, 
$cov_\ell(j(\bbar))=\sigma(cov_\ell(\bbar))$.  In particular, for each $q\in\Q$,
$j$ maps $cov_k^{-1}(\aa_q)$ onto $cov_k^{-1}(\aa_{h^*(q)})$ setwise.  
Thus, $j$ preserves $X^*$ setwise, and moreover maps each $X_n$ onto $X_{h(n)}$.  This means that $j$ also preserves $E^*$, and by our interpretation of
$R$ in both  $M_{G_1}$ and $M_{G_2}$, $j:M_{G_1}\rightarrow M_{G_2}$ is an $L'$-isomorphism.
%

$(4)\Rightarrow(2)$:  Suppose $M\models \phi$ is countable and $S_\infty$ divides $\Aut(M)$.  Choose an expansion $M^*$ of $M$ to an $L^*$-structure
for which there is
a continuous surjection $f:\Aut(M^*)\rightarrow Aut(\E)$ (recall that $\E=(\omega,=)$ and $S_\infty=\Aut(\omega,=)$).  
By Lemma~\ref{algebra}, there is an HC-reduction  $(\C_{\E},\sim_\conj)\le_{HC} (\C_{M^*},\sim_\conj)$.
Let $\phi_{M^*}$ be a Scott sentence for $M^*$.  Composing the above with the HC-reduction from
 Lemma~\ref{FSidea} and applying Lemma~\ref{bigcompare} to the right hand side yields an HC-reduction
 $$(\Mod_\omega(\mbox{Graphs}),\cong)\le_{HC} (\Mod_\omega(\phi_{M^*}^\flat\wedge\Axstar),\cong_{(L^*)^\flat})$$
 As $||(\mbox{Graphs},\cong)||=\infty$, it follows from Fact~\ref{bound} that there is a proper class of non-$(L^*)^\flat$-isomorphic models of $\phi_{M^*}^\flat\wedge\Axstar$.
Since $\phi_{M^*}^\flat\wedge\Axstar\vdash\phi_{M}^\flat\wedge\Ax$, the latter has arbitrarily large models.

%
%

$(2)\Rightarrow(4)$: $\phi^\flat\wedge \Ax$ has arbitrarily large models.
 By Fact~\ref{Morley}(3),  choose 
$L^*\supseteq \Lflat$ and a countable $L^*$-structure $\B^*\models\phi^\flat\wedge \Ax$
that is the Skolem hull of a sequence $\<\aa_q:q\in\Q\>$ of order indiscernibles and let $\B$ be the $\Lflat$-reduct of $\B^*$.
By Lemma~\ref{connectgroups}(1), $\Aut(\Q,\le)$ divides $\Aut(\B^*)$, hence also divides $\Aut(\B)$, since $\Aut(\B^*)\le \Aut(\B)$ is closed.
As $\B\models\phi^\flat$ is countable, choose $(M,\F)$ as in Proposition~\ref{reverse}.
 Then $M\models\phi$ and is countable.  By Lemma~\ref{connectgroups}(2) $\Aut(\B)$ divides $\Aut(M)$, hence  $\Aut(\Q,\le)$ divides $\Aut(M)$ by transitivity.
Thus $S_\infty$ divides $\Aut(M)$ by Lemma~\ref{QSsame}. 
\hfill\qedsymbol\end{proof}


\begin{Corollary}  \label{phiEM}  Suppose $\phi\in L_{\omega_1,\omega}$ admits Ehrenfeucht-Mostowski models.  Then $\phi$ has a Borel complete expansion.
\end{Corollary}

\begin{proof}  Apply Fact~\ref{Morley} directly to $\phi$ and get $L^*\supseteq L$ countable and a countable $L^*$-structure $M^*\models\phi$ that is the Skolem hull
of an order-indiscernible sequence $\<a_q:q\in\Q\>$.   By Lemma~\ref{connectgroups}(1), $\Aut(\Q,\le)$ divides $\Aut(M^*)$, hence also divides $\Aut(M)$,
where 
$M$ is the $L$-reduct of $M^*$. 
Thus, by Lemma~\ref{QSsame} $S_\infty$ divides $\Aut(M)$, so Theorem~\ref{HjorthThm} applies.
\hfill\qedsymbol\end{proof}

The converse of  Corollary~\ref{phiEM} does not hold.
It follows from Theorem~1.10 of \cite{herrings} that if $\phi_M$ is the Scott sentence of the Fra\"iss\'e limit  $M$ of a class $K$ of finite structures satisfying disjoint amalgamation,
then $S_\infty$ divides $\Aut(M)$.  
Thus, any of the examples of Scott sentences $\phi_n$ characterizing $\aleph_n$ 
in  \cite{BKL} do not admit Ehrenfeucht-Mostowski models, but have Borel complete expansions.  
Similarly,  in  \cite{HjorthKnight} Hjorth shows that for every $\alpha<\omega_1$, there is a Scott sentence
$\psi_\alpha$ characterizing $\aleph_\alpha$ with $S_\infty$ dividing the automorphism group of its countable model.   So again, each $\psi_\alpha$ does not admit Ehrenfeucht-Mostowski
models, yet has  a Borel complete expansion.  

By contrast, the much older example of  Knight in \cite{Knight}  of a complete $L_{\omega_1,\omega}$-sentence $\phi_{\Kn}$ characterizing $\aleph_1$  is  not obtained as a Fra\"iss\'e limit.
In \cite{HjorthKnight}, Hjorth proves that $S_\infty$ does not divide $\Aut(M)$, the automorphism group of the (unique) countable model of $\phi_{\Kn}$. 
%
Thus, $\phi_{\Kn}$ does not have a  Borel complete expansion by Theorem~\ref{HjorthThm}.

\section{Applications} We record two   instances where the conditions of Theorem~\ref{HjorthThm} hold.

\begin{Corollary}  \label{example}  Let $L$ be a countable language.
\begin{enumerate}
\item  If $T$ is a first-order $L$-theory with an infinite model, then  $T$ has a Borel complete expansion.
\item  If $\phi\in L_{\omega_1,\omega}$ and $||\phi_\ptl||\ge\beth_{\omega_1}$, then $\phi$ has a Borel complete expansion.
\end{enumerate}
\end{Corollary}

\begin{proof}  (1) By Ramsey and compactness, $T$ admits Ehrenfeucht-Mostowski models, so apply Corollary~\ref{phiEM} to $\phi:=\bigwedge T$.

For (2), it follows from $||\phi_\ptl||\ge\beth_{\omega_1}$ that $\phi^\flat\wedge\Ax$ has at least $\beth_{\omega_1}$ non-isomorphic models.
As $L^\flat$ is countable, for every $\alpha<\omega_1$, there is a model of  $\phi^\flat\wedge\Ax$ of size at least $\beth_\alpha$.  So, by applying
Fact~\ref{Morley} to $\phi^\flat\wedge\Ax$, there are arbitrarily large models of $\phi^\flat\wedge\Ax$.  Thus, Theorem~\ref{HjorthThm} applies.
\hfill\qedsymbol\end{proof}

%

More interesting are cases where the conditions of Theorem~\ref{HjorthThm} do not hold.  
In \cite{LU}, the authors considered families of cross-cutting equivalence relations, each with finitely many classes.  
More formally, fix the language $L=\{E_n:n\in\omega\}$.  
For any function $h:\omega\rightarrow(\omega-\{0,1\})$,  let $T_h^{\forall}$ be the universal $L$-theory asserting
that each $E_n$ is an equivalence relation with at most $h(n)$ classes.
It is easily checked that $T_h^\forall$ has a model completion $T_h$, which asserts that each $E_n$ has exactly $h(n)$ classes,
and the relations cross-cut, i.e., for any finite $F\subseteq \omega$, if we let $E_F(x,y):=\bigwedge_{n\in F} E_n(x,y)$, then $E_F$ has exactly
$\Pi_{n\in F} h(n)$ classes.  

In Remark~2.7 of \cite{LU}, the authors prove that if the set of integers $\{h(n):n\in\omega\}$ is unbounded, then $T_h$ is Borel complete.
Here, we use Theorem~\ref{HjorthThm} to obtain the converse.

\begin{theorem}  \label{bounded}  If the set of integers $\{h(n):n\in\omega\}$ is  bounded, then $T_h$ is not Borel complete.
\end{theorem}

\begin{proof}  
As $T_h$ is first-order, Corollary~\ref{example} shows we cannot apply Theorem~\ref{HjorthThm}
to  $T_h$ directly.  Rather, consider an arbitrary countable model $M\models T_h$.  
Consider the infinitary formula $E_\infty(x,y):=\bigwedge_{n\in\omega} E_n(x,y)$.  Then $E_\infty$ is an equivalence relation. 
For each $a\in M$,
$[a]_\infty:=\{b\in M: E_{\infty}(a,b)\}$, which ostensibly can have any cardinality $m$, $1\le m\le \omega$.
Let $\phi_h\in L_{\omega_1,\omega}$ assert:
$$\bigwedge T_h\wedge \forall x\forall y (E_\infty(x,y)\rightarrow x=y)$$
It is easily checked that any model $M\models\phi_h$ has cardinality at most $2^{\aleph_0}$.
Moreover, its automorphism group $\Aut(M)$ is easy to analyze.

\medskip\noindent{\bf Claim.}  For
$M\models\phi_h$, $\Aut(M)$  isomorphically embeds into  $\Pi_{n\in\omega} \Sym(h(n))$, where $\Sym(h(n))$ is the  group of permutations
of $\{0,\dots,h(n)-1\}$.
Since $\{h_n\}$ is  bounded, $\Aut(M)$
has bounded exponent.

\begin{proof}  
 As no two elements of $M$ realize the same $E_\infty$-class,  every automorphism of $M$ is determined by
how it permutes the classes.  As the equivalence relations $E_n$ have $h(n)$ classes,
 $\Aut(M)$ is isomorphic to a subgroup of  $\Pi_{n\in \omega} \Sym(h(n))$.  
Since $\{h(n)\}$ is bounded, choose  an integer $K$ so that $h(n)\le K$ for all $n\in\omega$. 
 As each of the finite groups
$\Sym(h(n))$ has exponent dividing $K!$, so does $\Aut(M)$.
\hfill\qedsymbol\end{proof}

It follows from the Claim that $S_\infty$ cannot divide the automorphism group of any countable model $M\models\phi_h$, so   
we conclude from Theorem~\ref{HjorthThm} that no expansion of $\phi_h$ can be Borel complete.

To finish the proof of Theorem~\ref{bounded}, we note there is a natural Borel reduction from $\Mod_\omega(T_h)$ to $\Mod_\omega(\psi_h)$ for some expansion $\psi_h\vdash\phi_h$.
Given a countable $M\models T_h$, we simply encode the size of each $E_\infty$-class by unary predicates.
In more detail, let  $L^*=L\cup\{U_m:1\le m\le \omega\}$, where each $U_m$ is a unary predicate, and let
$\psi_h$ be the $(L^*)_{\omega_1,\omega}$-sentence asserting $\phi_h$ and that the predicates $\{U_m:1\le m\le\omega\}$ partition the universe.
As it is an expansion of $\phi_h$, $\psi_h$ is not Borel complete.

But easily, $T_h$ and $\psi_h$ are Borel equivalent, since given any countable $M\models T_h$, let $f(M)$ be the $L^*$-structure
with  universe $M/E_\infty$, where the equivalence relations 
$E_n$ are interpreted naturally, and where $U_m([a]_{E_{\infty}})$ holds if and only if $|[a]_{E_{\infty}}|=m$ (i.e., we `color the branches by the number of realizations of the branch').
Since $\psi_h$ is not Borel complete, neither is $T_h$.
\hfill\qedsymbol\end{proof}

\subsection*{Acknowledgements}
Both authors partially supported
by NSF grants DMS-1855789 and DMS-2154101.
We are grateful to the anonymous referee for their extremely careful reading of  preliminary versions of this paper and for indicating several inaccuracies.

\end{document}